\documentclass[12pt,epsfig,amsfonts]{amsart} 
\usepackage{amsmath,amsthm,amssymb,amscd,epsfig,color}
\usepackage{graphicx}
\usepackage{mathrsfs}

\newtheorem*{theorema}{Theorem A}
\newtheorem*{theoremb}{Theorem B}
\newtheorem*{theoremc}{Theorem C}

\newtheorem*{corollary}{Corollary}

\newtheorem{prop}{Proposition}[section]
\newtheorem{lemma}[prop]{Lemma}

\theoremstyle{definition}

\newcommand{\bra}{[\![}
\newcommand{\ket}{]\!]}
\theoremstyle{remark}
\newtheorem{remark}[prop]{Remark}

\numberwithin{equation}{section}

\usepackage[utf8]{inputenc}
\usepackage{fancyvrb}

\usepackage{mathtools}
\usepackage{ifmtarg}

\begin{document}

\author{Hiroki Takahasi}

\address{Keio Institute of Pure and Applied Sciences (KiPAS), Department of Mathematics,
Keio University, Yokohama,
223-8522, JAPAN} 
\email{hiroki@math.keio.ac.jp}

\date{\today}

\title[Level-2 LDP for countable Markov shifts without Gibbs states]
{Level-2 
Large deviation principle for\\ countable Markov shifts without Gibbs states}

\begin{abstract}
We consider level-2 large deviations for the one-sided countable full shift without assuming
the existence of Bowen's Gibbs state.  
To deal with non-compact closed sets,
we provide a sufficient condition in terms of inducing
which ensures the exponential tightness of a sequence of Borel probability measures constructed from periodic configurations.
Under this condition we establish the level-2 Large Deviation Principle.
We apply our results to the continued fraction expansion
of real numbers in $[0,1)$ generated by the R\'enyi map,
and obtain the level-2 Large Deviation Principle, as well as a weighted equidistribution of a set of quadratic irrationals 
to equilibrium states of the R\'enyi map.
\end{abstract}
\maketitle

\section{Introduction}\label{intro}
 Dynamical systems (iterated maps)
equipped with finite Markov partitions are represented as finite Markov shifts, and
the construction of relevant invariant measures and the investigation of their statistical properties are done on the symbolic level, with adaptations of ideas 
in statistical mechanics (see e.g., \cite{Bow75,BowRue75,Rue78,Sin72,Wal78}).
This thermodynamic formalism initiated in the  60s
has been successfully extended 
to maps with infinite Markov partitions and shift spaces with countably infinite number of states 
(see e.g., \cite{ADU93,BS03,FieFieYur02,GS98,MauUrb03,Sar99,Sar03,Y99}). This paper is concerned with level-2 large deviations for such {\it countable Markov shifts},  
 and its application to a dynamical system related to number theory.

The theory of large deviations aims to characterize limit behaviors of probability measures in terms of rate functions. 
Let $\mathcal X$ be a topological space, and
let $\mathcal M(\mathcal X)$ denote the space
of Borel probability measures on $\mathcal X$ endowed with the weak* topology. 
We say a sequence 
 $\{\tilde\mu_n\}_{n=1}^\infty$ in $\mathcal M(\mathcal X)$
 satisfies the
{\it Large Deviation Principle} (LDP)
 if there exists a lower semicontinuous function $I\colon\mathcal X\to[0,\infty]$ such that 
\begin{equation}\label{LDPlow}
 \liminf_{n\to\infty}\frac{1}{n}\log \tilde\mu_n(\mathcal G)
\geq -\inf_{\mathcal G} I\ \text{ for any open set } 
\mathcal G\subset\mathcal X,\end{equation}
and
\begin{equation}\label{LDPup}
\limsup_{n\to\infty}\frac{1}{n}\log \tilde\mu_n(\mathcal C)\leq-\inf_{\mathcal C}I \ 
\text{for any closed set }\mathcal C\subset \mathcal X,\end{equation}
where $\inf\emptyset=\infty$, $\log0=-\infty$ by convention.
We call $x\in \mathcal X$
a minimizer if $I(x)=0$ holds. The set of minimizers is a closed set.
 The LDP means that in the limit $n\to\infty$ the measure
      $\tilde\mu_n$ assigns
      all but exponentially small mass 
      to the set of minimizers.
      The function $I$ is called a {\it rate function}, and called a {\it good rate function}
      if its level set $\{x\in\mathcal X\colon I(x)\leq\alpha\}$ is compact
for any $\alpha>0$.
If $\mathcal X$ is a metric space and $\{\tilde\mu_n\}_{n=1}^\infty$ satisfies the LDP, the rate function is unique.
   The setup in our mind is that
 $\mathcal X$ is the space of Borel probability measures on a topological space $X$ on which a Borel map $\sigma$ acts, and each 
 $\tilde\mu_n\in\mathcal M(\mathcal X)$ is given in terms of 
 empirical measures \[V_n(x)=\frac{1}{n}(\delta_x+\delta_{\sigma x}+\cdots\delta_{\sigma^{n-1}x}),\]
 where $\delta_{\sigma^kx}\in\mathcal X$ denotes the unit point mass at $\sigma^kx\in X$.
 We refer to the LDP in this setup as {\it level-2} \cite[Chapter~1]{Ell85}.


 For topologically mixing finite Markov shifts together with H\"older continuous potentials, the level-2 LDP for empirical distributions
 and that for sequences constructed from empirical measures on periodic orbits  were established in \cite{Kif90,OrePel89,Tak84} and \cite{Kif94} respectively.
A key ingredient 
in these classical cases is the existence of Bowen's Gibbs states \cite{Bow75}. 
With the aid of Bowen's Gibbs states, one can
 deduce the lower bound \eqref{LDPlow} 
by combining
Birkhoff's and Shannon-McMillan-Breiman's theorems, and the upper bound \eqref{LDPup} 
by modifying the standard proof of the variational principle \cite{Wal82} (see \cite{Tak84}).
For countable Markov shifts, Bowen's Gibbs states were constructed
 under the assumption of a good regularity of 
 potentials and a strong connectivity of transition matrices
 defining the shift spaces (see e.g., \cite{ADU93,BS03,GS98,MauUrb03,Sar99}). 
Several level-2 LDPs were established in \cite{Tak19} under the existence of Bowen's Gibbs states.
For a Markov chain on a countable state space,
a `weak level-2 LDP' restricting the upper bound to compact closed sets was obtained in
 \cite{FF02}.
 
It has been realized that
 not all dynamically relevant invariant probability measures correspond to Bowen's Gibbs states.
 One of the best known examples is
  the absolutely continuous invariant probability measure of an interval map of Manneville-Pomeau type, with finitely many branches and a neutral fixed point. 
Such a measure still retains a weak form of Bowen's Gibbs state \cite{Y99},
and has the weak Gibbsian property in statistical mechanics sense \cite{EFS93,MRTMV00,Yur03}.
For a thermodynamic formalism and level-2 large deviations for a class of this map, see \cite{Hof77,PS92,Y99}
and \cite{PS09,PSY98} respectively.
With these historical developments and the abundance of interesting dynamical systems 
 modeled by countable Markov shifts without Bowen's Gibbs states (see e.g., \cite{Iom10,Yur05}), it
 is important 
 to establish the level-2 LDP for 
 countable Markov shifts without assuming the existence of Bowen's Gibbs states.

A main new difficulty for countable Markov shifts
 is a treatment of non-compact closed sets.
 We say a sequence $\{\tilde\mu_n\}_{n=1}^\infty$ of Borel probability measures on a non-compact space $\mathcal X$ is
 {\it exponentially tight}  if for any $L>0$ there exists a compact set $\mathcal K\subset\mathcal X$
such that 
\[\limsup_{n\to\infty}\frac{1}{n}\log\tilde\mu_n(\mathcal X\setminus \mathcal K)\leq-L.\]
If $\{\tilde\mu_n\}_{n=1}^\infty$ is
  exponentially tight, then one only has to consider compact closed sets in verifying
the upper bound \eqref{LDPup},
  see e.g., \cite{DemZei98} for details.
 The proof of the
level-2 LDPs in \cite{Tak19}
relies on the existence of Bowen's Gibbs states in order to verify the exponential tightness. 

In the lack of Bowen's Gibbs states,
 our strategy is to
   use {\it inducing} to verify the exponential tightness. Inducing is a familiar procedure in ergodic theory originally considered in works by Kakutani, Rohlin and others, and was used in the construction of absolutely continuous invariant measures or Gibbs-equilibrium states (see e.g., \cite{A73,Bow79,PesSen08,P80}). An inducing scheme we use here is given by the first return map to an a priori fixed domain. 
     In terms of this inducing,
   we will formulate a sufficient condition  which ensures 
   the exponential tightness for the original system. A key concept is that of {\it local Gibbs states}, to be introduced in Section~\ref{ind-G}.

\subsection{Statements of results}
Throughout the rest of this paper, 
let $\mathbb N$ denote the discrete set of positive integers.
Let $X$ denote
the one-sided infinite Cartesian product topological space of $\mathbb N$, called a {\it countable full shift}.
The topology of $X$ has a base that consists of cylinders
\[[p_1\cdots p_n]=\{x=(x_n)_{n=1}^\infty\in X\colon x_k=p_k\text{ for }
 k\in\{1,2,\ldots, n\}\},\]
where $n\geq1$ and $p_1\cdots p_n\in\mathbb N^n$. This topology is metrizable with the metric $d_X(x,y)=
\exp\left(-\inf\{n\geq1\colon  x_n\neq y_n\}\right)$
where $\exp(-\infty)=0$ by convention. 
Let
 $\sigma$ denote the left shift acting on $X$ continuously: $(\sigma x)_n=x_{n+1}$ for $n\geq1$.

Let $\phi\colon X\to\mathbb R$ be a function, called a {\it potential}.
We say $\phi$ 
 is {\it acceptable} if it is uniformly continuous and satisfies
\[\sup_{p\in \mathbb N}\left(\sup_{[p]}\phi-\inf_{[p]}\phi\right)<\infty.\]
We say $\phi$
is {\it locally H\"older continuous} if there exist
  $C>0$ and $\alpha\in(0,1]$ such that for any $ p\in \mathbb N$ and all $x,y\in[p]$,
\[|\phi(x)-\phi(y)|\leq
Cd_X(x,y)^\alpha.\]
Clearly, if $\phi$ is locally H\"older continuous then it is acceptable.
For each $n\geq1$ we write
$S_n\phi$ for the Birkhoff sum 
$\sum_{k=0}^{n-1}\phi\circ \sigma^k$,
and
 introduce a {\it pressure} 
\[P(\phi)=\lim_{n\to\infty}\frac{1}{n}\log\sum_{p_1\cdots  p_n\in \mathbb N^n}
\sup_{[p_1\cdots  p_n]}\exp S_n\phi. \]
This limit exists by the sub-additivity, 
which is never $-\infty$ since $X$ is the full shift.

 Let $\phi\colon X\to\mathbb R$
be acceptable and satisfy $P(\phi)<\infty$.
We consider a sequence $\{\tilde\mu_n\}_{n=1}^\infty$ of Borel probability measures on $\mathcal M(X)$ given by
\begin{equation}\label{nu_n}\tilde\mu_n=\frac{1}{Z_n(\phi)}\sum_{x\in E_n} \exp (S_n\phi(x ))\delta_{V_n(x) },\end{equation}
where 
\[E_n=\{x\in X\colon\sigma^n x=x\},\]
and $\delta_{V_n(x)}$ denotes the unit point mass at $V_n(x)$,
and $Z_n(\phi)$ the normalizing constant. 
In dynamical systems terms, $E_n$ is the set of periodic points of period $n$.
In statistical mechanics terms, the measure $\tilde\mu_n$ is closely related to the
canonical ensemble subject to a periodic boundary condition.

An {\it inducing scheme} consists of a subset $X^*$ of $X$ of the form
\begin{equation}\label{xstar}X^*=X\setminus\bigcup_{p\in \mathbb N\cap[1,p^*) } [p],\end{equation}
where $p^*\geq2$, and a function $R\colon X^*\to\mathbb N\cup\{\infty\}$ given by 
\begin{equation}\label{def-R}R(x)=\inf\{n\geq1\colon \sigma^nx\in X^*\}.\end{equation}
Given an inducing scheme $(X^*,R)$, 
for each $k\in\mathbb N$ we write \[\{R=k\}=\{x\in X^*\colon R(x)=k\},\] and
 define an induced map 
\begin{equation}\label{sigmahat}\tau\colon \bigcup_{k=1}^\infty\{R=k\}\mapsto \sigma^{
R(x)}x\in X^*,\end{equation}
 and define an inducing domain 
\begin{equation}\label{xhat} \Sigma=\bigcap_{n=0}^\infty \tau^{-n}\left(\bigcup_{k=1}^\infty\{R=k\}\right).\end{equation}
In other words, $\tau$ is the first return map to $X^*$ and $\Sigma$
 is the domain on which $\tau$ can be iterated infinitely many times.
 We call 
 $(\Sigma,\tau|_{\Sigma})$ an {\it induced system}.
 Given a potential $\phi\colon X\to\mathbb R$, we introduce
a parametrized family of {\it induced potentials} $\Phi_{\gamma}\colon \Sigma\to\mathbb R$ ($\gamma\in\mathbb R$) by 
 \begin{equation}\label{ind-po}\Phi_{\gamma}(x)=
 S_{R(x)}\phi(x)-\gamma R(x).\end{equation}
  As shown in Section~\ref{symbolic}, there exists
 a countably infinite partition that topologically conjugates
the induced system to 
 the countable full shift. The local H\"older continuity of the induced potential $\Phi_\gamma$ and its pressure $P(\Phi_\gamma)$ are well-defined in terms of this conjugacy.
 Our main result is stated as follows.


\begin{theorema}[the level-2 LDP]
  Let
$\phi\colon X\to\mathbb R$ be acceptable and satisfy $P(\phi)<\infty$. Suppose there exists
an induced system for which
the induced potentials 
$\Phi_{\gamma}$, $\gamma\in\mathbb R$ 
are locally H\"older continuous, 
and 
  there exists $\gamma_0\in\mathbb R$ such that $P(\Phi_{\gamma_0})=0$. 
 Then
  $\{\tilde\mu_n\}_{n=1}^\infty$ 
  is exponentially tight and
satisfies the LDP with the good rate function.
\end{theorema}

Let us define the rate function in Theorem~A.
 Let $\mathcal M(X,\sigma)$ denote the set of $\sigma$-invariant elements of $\mathcal M(X)$
and let $\mathcal M_\phi(X,\sigma)=\{
 \mu\in\mathcal M(X,\sigma)\colon\int\phi {\rm d}\mu>-\infty\}$.
  Define 
  $ F_\phi \colon \mathcal M(X) \to [-\infty,0]$ by
\[ F_\phi(\mu)
=
\begin{cases}h(\mu)+\int\phi {\rm d}\mu-P(\phi) &\text{ if $\mu\in\mathcal M_\phi(X,\sigma)$},\\
-\infty&\text{ otherwise,}\end{cases}\]
where
 $h(\mu)\in[0,\infty]$ denotes the measure-theoretic entropy of $\mu$ with respect to $\sigma$. 
 Since $X$ is the full shift, $\phi$ is acceptable and $P(\phi)<\infty$, 
 $\sup\phi$ is finite (see \cite[Proposition~2.1.9]{MauUrb03}). By \cite[Theorem~2.1.7]{MauUrb03}, for each $\mu\in \mathcal M_\phi(X,\sigma)$ we have 
 $h(\mu)+\int\phi {\rm d}\mu\leq P(\phi)<\infty$,
 and so $h(\mu)<\infty$. 
 The variational principle \cite[Theorem~2.1.8]{MauUrb03} holds:
\[P(\phi)=\sup\left\{ h(\mu)+\int\phi {\rm d}\mu\colon
\mu\in\mathcal M_\phi(X,\sigma)\right\},\]
 A measure $\mu\in \mathcal M_\phi(X,\sigma)$ which attains this supremum is called an {\it equilibrium state} for the potential $\phi$.
The rate function $I_\phi\colon\mathcal M(X)\to[0,\infty]$ 
in Theorem~A is given by
\begin{equation}\label{ratefcn}
I_\phi(\mu)
=
-\inf_{\mathcal G \ni \mu}\sup_{\mathcal G}F_\phi,
\end{equation}
where the infimum is taken over all open subsets $\mathcal G$ of $\mathcal M(X)$ containing $\mu$.
Since the entropy is not upper semicontinuous on $\mathcal M(X,\sigma)$,
$I_\phi$ may not equal $-F_\phi$.

 Since the sequence $\{\tilde\mu_n\}_{n=1}^\infty$ in Theorem~A is exponentially tight, it is tight. 
 By Prohorov's theorem, it has a limit point. 
 Since the rate function $I_\phi$ in Theorem~A is the good rate function,
there exists at least one minimizer.
If the minimizer is unique,
 we obtain a 
 `level-2 weighted equidistribution
of elements of $\bigcup_{n=1}^\infty E_n$ toward minimizers'.

\begin{theoremb}[level-2 weighted equidistribution]
Let
$\phi\colon X\to\mathbb R$ be acceptable and satisfy $P(\phi)<\infty$. Suppose there exists
an induced system for which
the induced potentials 
$\Phi_{\gamma}$, $\gamma\in\mathbb R$ 
are locally H\"older continuous, and 
 there exists $\gamma_0\in\mathbb R$ such that $P(\Phi_{\gamma_0})=0$. 
 Assume that the minimizer of the rate function $I_\phi$ is unique, denoted by $\mu_{\rm min}$. 
For any bounded continuous function $\tilde\varphi\colon \mathcal M(X)\to\mathbb R$ we have
\[\lim_{n\to\infty}\frac{1}{Z_{n}(\phi)}\sum_{x\in E_{n} }\exp (S_{n}\phi( x))\tilde\varphi(V_n(x) )=\tilde\varphi(\mu_{\rm min}).\]
\end{theoremb}
Under the assumption of Theorem~A, minimizers are not always unique, and not always an equilibrium state.
A sufficient condition was given in \cite{Tak20} which ensures that minimizers are equilibrium states.

Taking various continuous functions $\tilde\varphi$ in Theorem~B, we obtain convergences of various time averages over the elements of $E_n$. 
Let $C(X)$ denote the set of real-valued bounded continuous functions on $X$.
\begin{corollary}[Inspired by Olsen {\cite[Section~1.1]{Ols03}}]
Under the assumption of Theorem~B, assume moreover the minimizer is unique, denoted by $\mu_{\rm min}$.
\begin{itemize}

\item[(a)] For all $\varphi,\psi\in C(X)$,
\[\begin{split}\lim_{n\to\infty}\frac{1}{Z_n(\phi)}\sum_{ x\in E_n }\exp (S_n\phi( x))\frac{1}{n^2}&S_n \varphi( x)S_n\psi(x)=\int \varphi {\rm d}\mu_{\rm min}\int \psi {\rm d}\mu_{\rm min}.\end{split}\]
   
     \item[(b)]
For $\varphi,\psi\in C(X)$
  with $\inf \psi>0$, 
\[\lim_{n\to\infty}\frac{1}{Z_n(\phi)}\sum_{x\in E_n }\exp(S_n\phi(x))\frac{S_n \varphi( x)}{ S_n\psi(x)}=\frac{\int \varphi {\rm d}\mu_{\rm min} }{\int \psi {\rm d}\mu_{\rm min} }.\]

\item[(c)] 
For $\pi_1,\pi_2\in C(X)$ and a bounded continuous function $f\colon\mathbb R\to\mathbb R$, 
\[\begin{split}\lim_{n\to\infty}\frac{1}{Z_n(\phi)}\sum_{ x\in E_n }\exp(S_n\phi( x))\frac{1}{n^2}&\sum_{k_1,k_2=0}^{n-1} f(\pi_1(\sigma^{k_1} x)+\pi_2(\sigma^{k_2} x))\\
&=\int f {\rm d}(\mu_{\rm min}\circ \pi_1^{-1}\otimes\mu_{\rm min}\circ \pi_2^{-1}),\end{split}\]
where $\otimes$ denotes the convolution.
\end{itemize}  
\end{corollary}
\begin{proof}
Apply Theorem~B to the bounded continuous functions $\mu\in\mathcal M(X)\mapsto\int \varphi {\rm d}\mu\int\psi {\rm d}\mu$, $\mu\in\mathcal M(X)\mapsto\int \varphi {\rm d}\mu/\int\psi {\rm d}\mu$,  
 $\mu\in\mathcal M(X)\mapsto\int f {\rm d}(\mu\circ\pi_1^{-1}\otimes\mu\circ\pi_2^{-1})$
 respectively.
\end{proof}

  \begin{figure}
\begin{center}
\includegraphics[height=4cm,width=5.5cm]{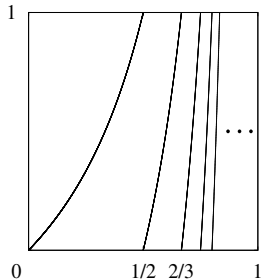}
\caption
{The graph of the R\'enyi map $T$.}\label{fig1}
\end{center}
\end{figure}


\subsection{Applications}\label{appl-intro}
Our results can be applied to dynamical systems modeled by the countable full shift without Bowen's Gibbs state. 
The assumption in Theorem~A can be verified, for example, for the infinite Manneville-Pomeau map \cite[Section~2.2]{Iom10},
and the two-dimensional conformal maps in \cite[Section~5]{Yur05} related to number theory. 
 Minimizers
of the associated rate functions are not unique, and so Theorem~B does not apply.
Further applications of different taste will be given in our forthcoming paper.

A prime
example to which our results apply
is the R\'enyi map
 $T\colon [0,1)\to [0,1)$ given by 
 \begin{equation}\label{renyi}
 T(\xi)=\frac{1}{1-\xi}-\left\lfloor \frac{1}{1-\xi}\right\rfloor,\end{equation}
     where $\lfloor \cdot\rfloor$ denotes the floor function.
      The graph of $T$ is obtained by reversing the graph of the well-known Gauss map $\xi\in(0,1]\to 1/\xi-\lfloor 1/\xi\rfloor\in[0,1)$ around the axis $\{\xi=1/2\}$, as shown in \textsc{Figure}~\ref{fig1}.
  The map $T$ 
 leaves invariant the absolutely continuous infinite measure ${\rm d}x/x$,
 and $x=0$ is its neutral fixed point: $T(0)=0$, $T'(0)=1$.
   The asymptotic distribution of typical orbits, in the Lebesgue measure sense, are concentrated on this neutral fixed point.

 The iteration of $T$ generates an infinite continued fraction expansion of each number  $\xi\in[0,1)$ of the form
\begin{equation}\label{CF}\xi=1-\cfrac{1 }{d_1(\xi)-\cfrac{1 }{d_2(\xi)-\ddots}},\end{equation}
where $d_n(\xi)=\lfloor 1/(1-T^{n-1}(\xi))\rfloor+1\geq2$ for $n\geq1$.
  The intervals
 \[J_p=\left[1-\frac{1}{p},1-\frac{1}{p+1}\right),\ p\in\mathbb N\]
form an infinite Markov partition of $[0,1)$, which allows us to represent $T$ as the left shift acting on $X$ \cite{Iom10,Tak22}.
 The coding map
 \begin{equation}\label{pi}\pi\colon (x_n)_{n=1}^\infty\in X\mapsto \pi((x_n)_{n=1 }^\infty)\in\bigcap_{n=1 }^\infty\overline{T^{-n+1}(J(x_{n}))}\subset[0,1)\end{equation}
 is well-defined 
and satisfies $T\circ \pi=\pi\circ\sigma$.
  We consider the potential
$\phi=-\log|T'\circ\pi|$, where $T'$ denotes the derivative of $T$ which is one-sided at boundary points of the Markov partition. From the mean value theorem applied to the inverse branches of $T$,
for any $p\geq1$ and all $\xi,\eta\in J_p$ we have
\[\log\frac{|T'(\xi)|}{|T'(\eta)|}\leq2|T(\xi)-T(\eta)|< 2.\]
In particular, $\phi$ is acceptable. Since
$\sup_{[p]}e^{\phi}$ is comparable to $p^{-2}$,
$P(\beta\phi)<\infty$ holds if and only if
 $\sum_{p\in\mathbb N} p^{-2\beta}$ is finite, which is equivalent to $\beta>1/2$.
It is easy to see that
for $n\geq1$,
 $T^{n}$ maps $[0,1/(n+1))$ diffeomorphically onto $[0,1)$. 
The mean value theorem implies $\lim_{n\to\infty}\sup_{[0,1/(n+1))} |(T^n)'|=\infty$, while
$|(T^n)'(0)|=1$ for $n\geq1$.
It follows that 
for any $\beta>1/2$
 there is no Bowen's Gibbs state for the potential $\beta\phi$.
Meanwhile, it is known \cite{Iom10} that for $\beta>1/2$,
the equilibrium state for $\beta\phi$ is unique, which we
denote by $\mu_{\beta\phi}$. 
 For $1/2<\beta<1$, $\mu_{\beta\phi}$ has positive entropy
and fully supported. For $\beta\geq1$, $\mu_{\beta\phi}$ is the unit point mass at $\pi^{-1}(0)$. 

 Let $\mathbb I$ denote the set of irrational numbers in $(0,1)$.
The set $E_n$ corresponds to
 the set of numbers in $\mathbb I\cup\{0\}$ for which the continued fraction 
  \eqref{CF} is periodic of period $n$.
 As in the proof of \cite[Theorem~28]{Khi64}, one can show that any number in  $\bigcup_{n=1}^\infty\{\xi\in\mathbb I\colon T^n(\xi)=\xi\}$ is a quadratic irrational, i.e., an irrational root of a quadratic polynomial
 with integer coefficients.
 Conversely, any quadratic irrational in $\mathbb I$ has an eventually 
 periodic continued fraction of the form \eqref{CF}, see \cite[Theorem~3]{Nor11}.
 
 An induced system as in Theorem~A is obtained from the first return map to the interval $[1/2,1)$ not containing the neutral fixed point.
From Theorems~A and B
we obtain the following.
For $\xi\in[0,1)$ and $n\geq1$, let $V_n(\xi)$ denote the empirical measure $(1/n)\sum_{k=0}^{n-1}\delta_{T^k(\xi)}$ on $[0,1)$.

\begin{theoremc}
For any $\beta\in(1/2,1]$,
the sequence of Borel probability measures on $\mathcal M(\pi(X))$ given by
\[\frac{1}{Z_n(\beta\phi)}
\sum_{\stackrel{\xi\in\mathbb I\cup\{0\}}{T^n(\xi)=\xi} } |(T^n)'(\xi)|^{-\beta}\delta_{V_n(\xi) }\ \text{ for }n=1,2,\ldots\]
satisfies the LDP. The minimizer is unique and it is the unit point mass at $\mu_{\beta\phi}\circ\pi^{-1}$. Moreover,
for any bounded continuous function $\tilde\varphi\colon\mathcal M(\pi(X))\to\mathbb R$ we have
\[\lim_{n\to\infty}\frac{1}{Z_n(\beta\phi) }
\sum_{\stackrel{\xi\in\mathbb I\cup\{0\} }{T^n(\xi)=\xi} }|(T^n)'(\xi)|^{-\beta}\tilde\varphi(V_n(\xi))=\tilde\varphi(\mu_{\beta\phi}\circ\pi^{-1}).\]
\end{theoremc}

\subsection{Structure of the paper}
The rest of this paper consists of two sections.
In Section~2 we verify the exponential tightness
of the sequence in \eqref{nu_n} under the assumption of Theorem~A.  In Section~3 we complete proofs of all the theorems.
We close with a remark on 
possible generalizations of the main results.

\section{Exponential tightness}
The aim of this section
is to verify the exponential tightness of the sequence in \eqref{nu_n}.
In Section~\ref{symbolic}
we start with a symbolic representation of the induced system. 
In Section~\ref{ind-G} we introduce the notion of local Gibbs states together with examples.
In Section~\ref{tight-section}
 we prove a main technical estimate
 assuming the existence of a local Gibbs state. 
Using this estimate, we verify the exponential tightness in Section~\ref{dedic}.
In Section~\ref{existence} we show that the assumption of Theorem~A implies the existence of a local Gibbs state.

\subsection{Symbolic representation of the induced system}\label{symbolic}
For a set $\mathbb S$ and an integer $j\geq1$,
let $\mathbb S^j$ denote the set of words of elements of $\mathbb S$ of word length $j$. We introduce an empty word $\emptyset$ and set $\mathbb S^0=\{\emptyset\}$,
 $a\emptyset=a=\emptyset a$, $a\emptyset b=ab$ for $a,b\in \mathbb S$. We set 
$W(\mathbb S)=\bigcup_{j=1}^\infty \mathbb S^j$,
$\mathbb N_0=\mathbb N\cup\{0\}$ and
$W_0(\mathbb S)=W(\mathbb S)\cup \mathbb S^0.$

Let $(X^*,R)$ be an inducing scheme. Set 
\[\mathbb N^*=\mathbb N\cap[p^*,\infty) 
\ \text{ and }\ \mathbb N_{*}=\mathbb N\cap[1,p^*).\]
For each $p\in \mathbb N^*$ and $\omega\in W_0(\mathbb N_*)$, the set $\bigcup_{q\in \mathbb N^* }[p \omega q]$ 
is mapped by the induced map  $\tau$ in \eqref{sigmahat}
bijectively onto $X^*$. Since the domain
$\bigcup_{k=1}^\infty\{R=k\}$ of definition of $\tau$ is partitioned into countably infinite sets of this form, the induced system $\tau|_{\Sigma}$ is represented as the countable full shift over the infinite alphabet
\begin{equation}\label{ahat}\mathbb A=\left\{\bigcup_{q\in \mathbb N^* }[p \omega q]\colon p\in \mathbb N^*\text{ and }\omega\in W_0(\mathbb N_*)\right\}.\end{equation}

To make the above statement into a rigorous one, we endow $\mathbb A$ with the discrete topology, and consider the countable full shift
\[{\mathbb A}^{\mathbb N}=\{z=(z_n)_{n=1}^\infty\colon z_n\in\mathbb A\text{ for   }n\geq1\}.\]
   We will often use  
 bold letters
to denote elements of $W(\mathbb A)$, and  a double square bracket $\bra\cdot\ket$ to denote cylinders in $\mathbb A^\mathbb N$:
the $n$-cylinder $(n\geq1)$ spanned by
 $\bold a=a_1\cdots a_n\in\mathbb A^n$ is 
 \[\bra\bold a\ket
 =\{z=(z_k)_{k=1}^\infty\in \mathbb A^\mathbb N\colon z_k=a_k\text{ for }
 k\in\{1,\ldots, n\}\}.\]
 By definition,
for each $k\in\{1,\ldots,n\}$ we have
    $a_k=\bigcup_{q\in \mathbb N^*}[p_k \omega^{(j_k)} q]$
where $p_k\in\mathbb N^*$, $j_k\in\mathbb N_0$, $\omega^{(j_k)}\in\mathbb N_{*}^{j_k}$. 
        Let $|\bold a|$ denote the word length of 
       the word $p_1\omega^{(j_1)}p_2\omega^{(j_2)}\cdots p_n\omega^{(j_n)}$ in $W(\mathbb N)$:
\[|\bold a|=n+\sum_{k=1}^nj_k.\]
Let $[\bold a]$ denote the corresponding $|\bold a|$-cylinder in $X$:
\[[\bold a]=[p_1\omega^{(j_1)}p_2\omega^{(j_2)}\cdots p_n\omega^{(j_n)}]\subset X.\]
It is easy to check that a coding map
$\Pi\colon\mathbb A^{\mathbb N}\to\Sigma$ is defined by
    \begin{equation}\label{codingmap}\Pi\colon (z_n)_{n=1}^\infty\in\mathbb A^{\mathbb N}\mapsto\Pi( (z_n)_{n=1}^\infty)\in\bigcap_{n=1}^\infty[ z_1\cdots z_n]\subset\Sigma,\end{equation}
    and becomes a homeomorphism.
    Let $\theta$ denote the left shift acting on ${\mathbb A}^{\mathbb N}$. Clearly we have
    $\Pi\circ\theta=\tau|_{\Sigma}\circ\Pi$, $\Pi(\bra\bold a\ket)=[\bold a]$ for any $\bold a\in W(\mathbb A)$.
    Let $d_{\Sigma}$ denote the metric on $\Sigma$ 
given by $d_\Sigma(x,y)=d_{\mathbb A^\mathbb N}(\Pi^{-1}(x),\Pi^{-1}(y))$.

For $\bold a\in\mathbb A^n$ as above
and $q\in\mathbb N^*$, we will often write \[\bold a q=p_1\omega^{(j_1)}p_2\omega^{(j_2)}\cdots p_n\omega^{(j_n)}q,\] 
 and interpret $\bold a q$ as a word in $W(\mathbb N)$ of word length $|\bold a|+1$.

\begin{lemma}\label{disjoint}
Let $(\Sigma,\tau|_{\Sigma})$ be an induced system and let $\Pi$ be the coding map in \eqref{codingmap}.
\begin{itemize}
\item[(a)] For every $\bold a\in W(\mathbb A)$,
\[\Pi[\![\bold a]\!]=
\Sigma\cap\bigcup_{q\in\mathbb N^*}[\bold a q].\]

\item[(b)] If $\bold a$, $\bold b\in W(\mathbb A)$ satisfy $|\bold a|=
|\bold b|$, then 
$\bold a=\bold b$ or 
$\bra\bold a\ket\cap\bra \bold b\ket=\emptyset.$
\end{itemize}
\end{lemma}
\begin{proof}
 If $n\geq1$, $\bold a\in\mathbb A^n$ 
   then the sum
   $\sum_{k=0}^{n-1}R\circ\tau^k$ of return times equals $|\bold a|$ on $\Pi\bra\bold a\ket$, which implies (a).
 If $\bold a$, $\bold b\in W(\mathbb A)$ satisfy
 $|\bold a|=
|\bold b|$ and
 $\bold a\neq\bold b$, then (a) implies 
 $\Pi\bra\bold a\ket\cap\Pi\bra \bold b\ket=\emptyset$, and
so
$\bra\bold a\ket\cap\bra \bold b\ket=\emptyset$, verifying (b).
\end{proof}

\subsection{Local Gibbs states}\label{ind-G}
 Let
$\phi\colon X\to\mathbb R$ and let $(\Sigma,\tau|_{\Sigma})$ be an induced system.
A Borel probability measure $\lambda_\phi$ on ${ \mathbb A}^{\mathbb N}$ is called
a {\it local Gibbs state for the potential $\phi$ associated with $(\Sigma,\tau|_{\Sigma})$}, if 
there exist constants $C\geq1$, $\gamma_0\in\mathbb R$ such that
for any
$\bold a\in  W(\mathbb A)$ and any
$x\in\Pi[\![\bold a]\!]$ we have
\begin{equation}\label{induce-gibbs}C^{-1}\leq\frac{\lambda_\phi[\![\bold a]\!]}{\exp\left(S_{|\bold a| }\phi(x)-\gamma_0|\bold a|\right)}\leq C.\end{equation}
We do not require the $\theta$-invariance of $\lambda_\phi$.
 We will simply call $\lambda_\phi$ a local Gibbs state (associated with $(\Sigma,\tau|_{\Sigma})$)
 if the context is clear.

\begin{remark}[Examples of local Gibbs states]\label{rem-renyi}
The 
R\'enyi map $T$ in \eqref{renyi} is modeled on the countable full shift $X$ via the coding map $\pi$ in \eqref{pi}.
In consideration of the neutral fixed point $0$ of $T$,
we set $p^*=2$ and
define an inducing scheme $(X^*,R)$ by
\eqref{xstar}, \eqref{def-R}, and
the induced system $(\Sigma,\tau|_{\Sigma})$ by  \eqref{sigmahat}, \eqref{xhat}. We also define the infinite alphabet
$\mathbb A$ and the coding map $\Pi$  by \eqref{ahat}, \eqref{codingmap} respectively. Note that
 $\Sigma=\left(1/2,1\right)\cap\mathbb I$.
Let $\mu_1$ denote the normalized restriction of the Lebesgue measure on $\mathbb R$ to $(1/2,1)$.
Let $\mu_2$ denote the normalized restriction of ${\rm d}x/x$ to $(1/2,1)$. Both $\mu_1\circ\pi\circ\Pi$ and $\mu_2\circ\pi\circ\Pi$ are local Gibbs states for the potential $\phi=-\log|T'\circ\pi|$ with $\gamma_0=0$. The latter is a $\theta$-invariant measure.\end{remark}

If $\lambda_\phi$ is a local Gibbs state, then for any $\bold a\in W(\mathbb A)$, the $\lambda_\phi$-measure of the cylinder $\bra\bold a\ket$ in $\mathbb A^\mathbb N$ is given (up to multiplicative constants) by the Birkhoff sum of $\phi$
along the orbit of $x$ of length $|\bold a|$ and the word length $|\bold a|$. 
Since the word length of $\bold a$
as a word in $W(\mathbb A)$ does not appear in the formula \eqref{induce-gibbs}, $\lambda_\phi$ well captures part of the original dynamics $(X,\sigma)$.

\begin{lemma}\label{spread}
 Let $\phi\colon X\to\mathbb R$, let $(\Sigma,\tau|_{\Sigma})$ be an induced system and let $\lambda_\phi$ be a local Gibbs state for the potential $\phi$ associated with $(\Sigma,\tau|_{\Sigma})$. 
 For any $\bold a\in W(\mathbb A)$ we have
\[\lambda_\phi[\![\bold a]\!]=\lambda_\phi\circ\Pi^{-1}(\Sigma\cap[\bold a]).\]
\end{lemma}
\begin{proof}
We write $\nu_\phi$ for
$\lambda_\phi\circ\Pi^{-1}$.
 We have
$\nu_\phi\{R=n\}>0$ for every $n\geq1$. Let $\nu_\phi|_{\{R=n\}}$ denote the restriction of $\nu_\phi$ to $\{R=n\}$.
The measure
\[\mu_\phi=\sum_{n=1}^\infty\sum_{k=0}^{n-1}\nu_\phi|_{\{R=n\} }\circ\sigma^{-k}\]
is a finite measure if and only if $\int R{\rm d}\nu_\phi<\infty$.
Since the restriction of $R$ to $\Sigma$ is the first return time to $\Sigma$, 
$\{R=n\}$ is disjoint from $\Sigma\cap\bigcup_{k=1}^{n-1}\sigma^{-k}(\Sigma)$ for $n\geq2$ and we have
\begin{equation}\label{measure-lem0}\mu_\phi|_{\Sigma}=
\sum_{n=1}^\infty
 \nu_\phi|_{\{R=n\}}=\nu_\phi=\lambda_\phi\circ\Pi^{-1}.\end{equation}
  For any 
 $\bold a\in  W(\mathbb A)$ we have
 $\Pi[\![\bold a]\!]\subset \Sigma$, and so
\begin{equation}\label{measure-lem10}\lambda_\phi[\![\bold a]\!]=\mu_\phi\Pi[\![\bold a]\!].\end{equation}
 By 
$\mu_\phi|_{\Sigma}=\nu_\phi$ in \eqref{measure-lem0} and Lemma~\ref{disjoint}(a), 
for any $\bold a\in W(\mathbb A)$ we have
\begin{equation}\label{measure-lem1}\mu_\phi\Pi[\![\bold a]\!]=\nu_\phi\Pi[\![\bold a]\!]=
\sum_{q\in\mathbb N^*}\nu_\phi(\Sigma\cap[\bold a q]).\end{equation}
Let $j\geq1$ satisfy $\bold a\in\mathbb A^j$.
Then
$\sigma^{|\bold a|}$ and
$\tau^j$ coincide on $\Sigma\cap[\bold a]$.
By $\tau(\Sigma)\subset\Sigma$ we have
 $\bigcup_{q\in\mathbb N_*}(\Sigma\cap[\bold a q])\subset(\tau|_{\Sigma})^{-j}(\bigcup_{q\in\mathbb N_*}[q])=\emptyset$, and so
 $\nu_\phi(\bigcup_{q\in\mathbb N_*}\Sigma\cap[\bold a q])=0.$
Combining this with \eqref{measure-lem10}, \eqref{measure-lem1} we obtain
\[\lambda_\phi[\![\bold a]\!]=\sum_{q\in\mathbb N_*}\nu_\phi(\Sigma\cap[\bold a q])+\sum_{q\in\mathbb N^*}\nu_\phi(\Sigma\cap[\bold a q])=\nu_\phi(\Sigma\cap[\bold a]),\]
as required.
\end{proof}

\subsection{Exponential decay on partition functions}\label{tight-section}
The next proposition provides a main technical estimate
under the existence of a local Gibbs state.

\begin{prop}\label{expo}
Let $\phi\colon X\to\mathbb R$, let $(\Sigma,\tau|_{\Sigma})$ be an induced system and let $\lambda_\phi$ be a local Gibbs state for the potential $\phi$ associated with $(\Sigma,\tau|_{\Sigma})$.
There exist $\delta'\in(0,1/5]$ and $n_0\geq1$ such that if $\delta\in(0,\delta']$ and $\{N_i\}_{i=1}^\infty$ 
is a non-decreasing sequence of integers
such that
\begin{equation}\label{N1}\max\mathbb N_*\leq N_1 \text{ and }\end{equation} 
  \begin{equation}\label{deltas}
 \sum_{k=N_i+1}^\infty\sum_{\stackrel{ \bold a\in\mathbb A }
 {\Pi\bra\bold a\ket\subset[k]}} \lambda_\phi\bra  \bold a \ket\leq\delta^{2i} \text{ for every }i\geq1,\end{equation}
then for every 
$n\geq n_0$ and every $m\in\{1,\ldots,n\}$ we have
\begin{equation}\label{desired-ineq}
\sum_{\stackrel{x\in E_n }
{V_n(x)(X\setminus\Gamma)=m/n}}
\exp S_n\phi(x)
\leq \frac{2^nne^{\gamma_0n}(4\delta)^m}{1-4\delta},\end{equation}
 where \begin{equation}\label{Gamma}\Gamma=\{x=(x_i)_{i=1}^\infty\in X\colon x_i\leq N_i \text{ for every }i\geq1\}.\end{equation}
\end{prop}
Proposition~\ref{expo} asserts that the
contribution of elements of $E_n$ to 
$Z_n(\phi)$ whose orbits escape from the compact set $\Gamma$ exactly $m$ times within period $n$ is exponentially small in $m$.
Similar estimates were obtained in \cite{Tak19} under the existence of Bowen's Gibbs states.
\begin{proof}[Proof of Proposition~\ref{expo}]
Since $\lambda_\phi$ is a local Gibbs state, 
there exist constants $C\geq1$ and $\gamma_0\in\mathbb R$ such that for any
$\bold a\in  W(\mathbb A)$ and any
$x\in\Pi[\![\bold a]\!]$ we have
\begin{equation}\label{equation1}C^{-1}\leq\frac{\lambda_\phi[\![\bold a]\!]}{\exp\left(S_{|\bold a| }\phi( x)-\gamma_0|\bold a|\right)}\leq C.\end{equation}
For the rest of the proof of Proposition~\ref{expo}, we will use the notation
 $a\ll b$ 
for two positive reals $a$, $b$
to indicate that 
 $a/b$ is bounded from above by a constant which depends only on $C$ in \eqref{equation1}. 
If $a\ll b$ and $b\ll a$, we will write
$a\asymp b$.

The first inequality in \eqref{equation1} will be used 
to bound a partial sum of $Z_n(\phi)$ 
from above by a sum of $\lambda_\phi$-measures
of cylinders in ${\mathbb A}^{\mathbb N}$.
Further, we will bound this sum using
the following product property
which is a consequence of \eqref{equation1}:
\begin{equation}\label{equation2}
\lambda_\phi[\![\bold a\bold b]\!]\asymp\lambda_\phi[\![\bold a]\!]\lambda_\phi[\![\bold b]\!] \text{ for  }\bold a,\bold b\in W(\mathbb A).\end{equation}

Let
$\delta\in(0,1/5]$ and let $\{N_i\}_{i=1}^\infty$ be a non-decreasing integer sequence satisfying 
\eqref{N1} and \eqref{deltas}. Let $\Gamma=\Gamma(\{N_i\}_{i=1}^\infty)$ be the compact subset of $X$ given by \eqref{Gamma}.
Let $n\geq1$.
If $x\in E_n\setminus \bigcup_{i=0}^{n-1}\sigma^{-i}(\Sigma)$ then
$x_i\leq \max\mathbb N_*\leq N_1\leq N_i$ for  $1\leq i\leq n$, and so
$V_n(x)(\Gamma)=1$.
By this and the periodicity of elements of $E_n$, for $1\leq m\leq n$ we have
\begin{equation}\label{step1}\begin{split}\sum_{\stackrel{x\in E_n }
{V_n(x)(X\setminus\Gamma)=m/n }}\exp S_n\phi(x)=&
\sum_{\stackrel{x\in E_n \cap\bigcup_{i=0}^{n-1}\sigma^{-i}(\Sigma)}
{V_n(x)(X\setminus\Gamma)=m/n }}\exp S_n\phi(x)\\
&+\sum_{\stackrel{x\in E_n \setminus\bigcup_{i=0}^{n-1}\sigma^{-i}(\Sigma)}
{V_n(x)(X\setminus\Gamma)=m/n }}\exp S_n\phi(x)\\
=&\sum_{\stackrel{x\in E_n \cap\bigcup_{i=0}^{n-1}\sigma^{-i}(\Sigma)}
{V_n(x)(X\setminus\Gamma)=m/n }}\exp S_n\phi(x)\\
\leq&
n\cdot\sum_{\stackrel{x \in E_n\cap \Sigma }
{V_n(x)(X\setminus\Gamma)=m/n }}\exp S_{n}\phi(x).\end{split}\end{equation}
In order to bound the last sum in \eqref{step1},
 we decompose the set $\{x\in E_n\cap \Sigma\colon
 V_n(x)
 (X\setminus\Gamma)=m/n\}$ into subsets sharing the same itinerary `up to time $n$', and estimate a contribution from each subset separately, and finally unify all these estimates counting the total number of possible itineraries.

Define a function $r\colon X\setminus\Gamma\to\mathbb N$ by
\[r(x)=\min\{i\geq1\colon x_i>N_i\}.\]
Let $x\in E_n\cap\Sigma$ satisfy $V_n(x)(X\setminus\Gamma)>0$.
By an {\it itinerary} of $x$ we mean a pair of  sequences $\{n_j( x)\}_{j=1}^\infty$ 
and $\{r_j( x)\}_{j=1}^\infty$
in $\mathbb N_0$ given by the recursion formulas
\[\begin{split}n_1(x)&=\min\{i\geq0\colon \sigma^{i}x\notin \Gamma\},\text{ and }\\
 r_j(x)&=r(\sigma^{n_j(x)}x),\ n_{j+1}(x)=
 \min\{i\geq n_j(x)+r_j(x)\colon \sigma^{i}x\notin \Gamma\}
  \text{ for }j\geq1.
 \end{split}\]
  Note that
 \[0\leq n_1(x)\leq\cdots \leq n_j(x)\leq n_j(x)+r_j(x)-1<n_{j+1}(x)\leq\cdots\]

  \begin{lemma}\label{gamma-easy}
  Let $x\in E_n\cap\Sigma$ satisfy $V_n(x)(X\setminus\Gamma)>0$.
\begin{itemize}
\item[(a)] $\{i\geq0\colon \sigma^i x\notin \Gamma\}=\bigcup_{j=1}^\infty[n_j(x),n_j(x)+r_j( x)-1]\cap\mathbb N_0.$
\item[(b)] $x_{n_j(x)+r_j(x)}\in\mathbb N^*$ for every $j\geq1$.
\item[(c)]
 If
  $n_j(x)\leq n-1$ then $ n_j(x)+ r_j(x)\leq 2n-1$.
\end{itemize}
\end{lemma}
\begin{proof}
Since $\{N_i\}_{i=1}^\infty$ is non-decreasing, 
 if $x\notin \Gamma$ then 
  $\sigma^ix\notin\Gamma$ for $0\leq i\leq r(x)-1$. This implies (a).
Since $\sigma^{n_j(x)} x=x_{n_j(x)+1}x_{n_j(x)+2}\cdots$ and  $\sigma^{n_j(x)}x\notin\Gamma$ with $r(\sigma^{n_j(x)}x)=r_j(x)$, we obtain
 $x_{n_j(x)+r_j(x)}>N_{r_j(x)}\geq N_1$. This together with \eqref{N1} yields
 $x_{n_j(x)+r_j(x)}\in\mathbb N^*$ as in (b).
 Since $x\in E_n$, (c) holds.
\end{proof}

  For each $j\in\{1,\ldots,m\}$,
 $n_1\cdots n_j\in\mathbb N_0^j$
with $0\leq n_1<\cdots <n_j< n$ and
 $r_1\cdots r_j\in\mathbb N^j$,
   we consider the set
 \begin{equation}\label{Delta}\Delta_{n_1\cdots n_j}^{r_1\cdots r_j}=\{ x\in E_n\cap \Sigma\colon (n_i( x),r_i(x) )=(n_i,r_i)\ \text{ for every }i\in\{1,\ldots,j\}\}.\end{equation}
 \begin{lemma}\label{measure}
 If $\delta\in(0,1/5]$ is sufficiently small,  then for all $j\in\{1,\ldots,m\}$, 
  $n_1\cdots n_j\in\mathbb N_0^j$ with $0\leq n_1<\cdots <n_j< n$ and 
  $r_1\cdots r_j\in\mathbb N^j$ we have
\[
\sum_{x\in\Delta_{n_1\cdots n_j}^{r_1\cdots r_j}}\exp S_{n}\phi(x)\leq e^{\gamma_0n}\delta^{r_1+\cdots +r_j}.\]
\end{lemma}
\begin{proof}
We may assume $\Delta_{n_1\cdots n_j}^{r_1\cdots r_j}\neq\emptyset$
for otherwise there is nothing to prove. Starting with
the case $j=1$, we introduce two sets of induced words
 \[\begin{split}
W_{0,1}&=\{\bold a\in W(\mathbb A)\colon|\bold a|=n_1+r_1-1\}\ \text{ and }\\
W_{0,2}&=\{\bold b\in W(\mathbb A)\colon|\bold b|=n-n_1-r_1+1,\
\Pi\bra\bold b\ket\subset\cup_{k=N_{r_1}+1}^\infty[k]\}.\end{split}\]
 For each $x\in
    \Delta_{n_1}^{r_1}$ we have $x_{n+1}=x_1\in\mathbb N^*$ by the definition \eqref{Delta}, and
    $x_{n_1+r_1}\in\mathbb N^*$ by Lemma~\ref{gamma-easy}(b).
Hence
    there exist $\bold a\in W_{0,1}$ and $\bold b\in W_{0,2}$ such that
$[\bold a]=[x_1\cdots x_{n_1+r_1-1}]$ and
$[\bold b]=[x_{n_1+r_1}\cdots x_n]$, and 
$x\in\Pi\bra\bold a\bold b\ket$. By
\eqref{equation1} and \eqref{equation2} we have
\[
\exp S_n\phi(x)\ll
 e^{\gamma_0n}\lambda_\phi\bra\bold a\bold b\ket\ll
 e^{\gamma_0n}\lambda_\phi\bra\bold a\ket\lambda_\phi \bra\bold b\ket.\]
 Summing this inequality over all
$x\in\Delta_{n_1}^{r_1}$,
and then using 
$\sum_{\bold a\in W_{0,1} }\lambda_\phi \bra\bold a\ket \leq 1$ and $\sum_{\bold b\in W_{0,2} }\lambda_\phi \bra\bold b\ket \leq \delta^{2r_1}$ which follow from Lemma~\ref{disjoint}(b) and
\eqref{deltas} respectively, we obtain
 \begin{equation}\label{i-step}\begin{split}\sum_{x\in\Delta_{n_1}^{r_1}  } \exp S_{n}\phi(x)\ll &e^{\gamma_0n}\sum_{\bold a
 \in W_{0,1} }\lambda_\phi\bra\bold a\ket
 \sum_{\bold b\in W_{0,2} }\lambda_\phi \bra\bold b\ket
 \leq e^{\gamma_0n}\delta^{2r_1}\leq e^{\gamma_0n}\delta^{r_1},\end{split}
 \end{equation}
 provided $\delta$ is small enough.
  In case $m=1$ we are done.  

To proceed, suppose $m\geq2$.
let $j,j+1\in\{1,\ldots,m\}$ and let
 $n_1\cdots n_jn_{j+1}\in\mathbb N_0^{j+1}$,
  $r_1\cdots r_jr_{j+1}\in\mathbb N^{j+1}$ satisfy
 $\Delta_{n_1\cdots n_jn_{j+1}}^{r_1\cdots r_jr_{j+1}}\neq\emptyset$.
 Define
 \[\begin{split}
 W_{j,1}&=\left\{\bold a\in W(\mathbb A)\colon|\bold a|=n_{j+1}+r_{j+1}-1\right\},\\
 W_{j,2}&=\{\bold b\in W(\mathbb A)\colon|\bold b|=n-n_{j+1}-r_{j+1}+1,\
\Pi\bra\bold b\ket\subset\cup_{k=N_{r_{j+1}}+1}^\infty[k]\},\\
  W_{j,3}&=\{\bold c\in W(\mathbb A)\colon|\bold c|=n_j+r_j,\ \Pi\bra\bold c\ket\cap\Delta_{n_1\cdots n_{j+1}}^{r_1\cdots r_{j+1} }\neq\emptyset\},\\
W_{j,4}&=\left\{\bold d\in  
W(\mathbb A)\colon|\bold d|=n-n_j-r_j\right\}.
\end{split}\]

  Let $\bold c\in W_{j,3}$.
For each $\bold d\in W_{j,4}$ we have $|\bold c\bold d|=n$.
Since $X$ is the full shift,
$\Pi\bra\bold c\bold d\ket$
contains a unique element of $E_n\cap \Sigma$ which we denote by $\overline{\bold c\bold d}$. We have $\overline{\bold c\bold d}\in\Pi\bra\bold c\ket\cap\Delta_{n_1\cdots n_j}^{r_1\cdots r_j}$, and by \eqref{equation1} and \eqref{equation2},
\[\exp S_n\phi(\overline{\bold c\bold d} )\gg e^{\gamma_0n}\lambda_\phi\bra\bold c\ket\lambda_\phi\bra\bold d\ket,\]
 which implies
\begin{equation}\label{expo-eq50}\begin{split}\sum_{x\in
\Pi\bra\bold c\ket\cap\Delta_{n_1\cdots n_j}^{r_1\cdots r_j} }\exp S_n\phi(x)
&\gg e^{\gamma_0n}\lambda_\phi\bra\bold c\ket\sum_{\bold d\in W_{j,4} }
\lambda_\phi\bra\bold d \ket.\end{split}\end{equation}
By Lemma~\ref{spread},
for each $\bold d\in W_{j,4}$ we have
$\lambda_\phi\bra\bold d\ket=\lambda_\phi\circ\Pi^{-1} (\Sigma\cap[\bold d]).$
 Since the sets $\Sigma\cap[\bold d]$,
 $\bold d\in W_{j,4}$ are pairwise disjoint and their union equals $\Sigma$, we have
\begin{equation}\label{boldc}\sum_{\bold d\in W_{j,4}}\lambda_\phi\bra\bold d\ket=1.
\end{equation}
Combining \eqref{expo-eq50} and \eqref{boldc} yields
\begin{equation}\label{expo-eq60}\begin{split}\sum_{x\in\Pi\bra\bold c\ket\cap\Delta_{n_1\cdots n_j}^{r_1\cdots r_j} }\exp S_n\phi(x)
\gg e^{\gamma_0n}\lambda_\phi\bra\bold c\ket.\end{split}\end{equation}

Similarly to the case $j=1$, for each
$x\in
    \Pi\bra\bold c\ket\cap  \Delta_{n_1\cdots n_{j+1}}^{r_1\cdots r_{j+1}}$ we have $x_{n+1}=x_1\in\mathbb N^*$ by the definition \eqref{Delta} and
    $x_{n_{j+1}+r_{j+1}}\in\mathbb N^*$ by Lemma~\ref{gamma-easy}(b).
    Hence
    there exist $\bold a\in W_{j,1}$, $\bold b\in W_{j,2}$
such that
$[\bold a]=[x_1\cdots x_{n_{j+1}+r_{j+1}-1}]$ and
$[\bold b]=[x_{n_{j+1}+r_{j+1}}\cdots x_n]$.
We have
$\bra\bold a\ket\subset\bra\bold c\ket$ and
$x\in\Pi\bra\bold a\bold b\ket$, and by \eqref{equation1} and \eqref{equation2},
\[\exp S_n\phi(x)\ll
e^{\gamma_0n}\lambda_\phi\bra\bold a\ket\lambda_\phi \bra\bold b\ket.\]
 Summing this inequality over all $x\in
    \Pi\bra\bold c\ket\cap  \Delta_{n_1\cdots n_{j+1}}^{r_1\cdots r_{j+1}}$, and
    then using 
$\sum_{\stackrel{\bold a\in W_{j,1} }{\bra\bold a\ket\subset\bra\bold c\ket}}\lambda_\phi\bra\bold a\ket\leq\lambda_\phi\bra\bold c\ket$ 
and $\sum_{\bold b\in W_{j,2} }\lambda_\phi \bra\bold b\ket \leq \delta^{2r_{j+1}}$ which follow 
from Lemma~\ref{disjoint}(b) and  \eqref{deltas} respectively,
we obtain
\begin{equation}\label{expo-eq51}\begin{split}
    \sum_{x\in
    \Pi\bra\bold c\ket\cap  \Delta_{n_1\cdots n_{j+1}}^{r_1\cdots r_{j+1}} }\exp S_n\phi(x)
&\ll e^{\gamma_0n}\sum_{\stackrel{\bold a\in W_{j,1} }{\bra\bold a\ket\subset\bra\bold c\ket}}\lambda_\phi\bra\bold a\ket
\sum_{\bold b\in W_{j,2} }
\lambda_\phi\bra\bold b\ket\\
&\leq
e^{\gamma_0n}\lambda_\phi\bra\bold c\ket\delta^{2r_{j+1}}.
 \end{split}\end{equation}
 Combining \eqref{expo-eq60} and \eqref{expo-eq51} yields 
\[\frac{\sum_{x\in\Pi\bra\bold c\ket\cap   \Delta_{n_1\cdots n_{j+1}}^{r_1\cdots r_{j+1}} }\exp S_n\phi(x )}{
\sum_{x\in\Pi\bra\bold c\ket\cap  \Delta_{n_1\cdots n_j}^{r_1\cdots r_j} }\exp S_n\phi(x)}\leq\delta^{r_{j+1}},\]
provided 
$\delta$ is small enough.
Rearranging this inequality and summing the result over all $\bold c\in W_{j,3}$ yields 
\[\sum_{x\in \Delta_{n_1\cdots n_{j+1}}^{r_1\cdots r_{j+1}}  }\exp S_n\phi(x )\leq\delta^{r_{j+1}}\sum_{x \in \Delta_{n_1\cdots n_j}^{r_1\cdots r_j}  }\exp S_n\phi(x ).\]
Applying this inequality recursively and combining the final result with \eqref{i-step} yields the desired inequality in Lemma~\ref{measure}.\end{proof}

Returning to the proof of Proposition~\ref{expo}, 
for two integers $R\geq m$ and $s\in\{1,\ldots,m\}$ we denote by $K_{R,s}$
the cardinality of the set of elements $(n_1\cdots n_s,r_1\cdots r_s)$ of
$\mathbb N_0^s\times\mathbb N^s$ with
$0\leq n_1<\cdots< n_s< n$ and $r_1+\cdots+r_s=R$.
The number of ways of locating $n_1,\ldots,n_s$
in $[0,n]$ does not exceed
$\left(\begin{smallmatrix}n\\s\end{smallmatrix}\right)$, and
for each location $(n_1,\ldots,n_s)$
the number of all feasible combinations of  $(r_1,\ldots,r_s)$ with $r_1+\cdots+r_s=R$ is bounded by the number of ways
of dividing $R$ objects into $s$ groups, not exceeding
$\left(\begin{smallmatrix}R+s-1\\s-1\end{smallmatrix}\right)\leq 2^{R+s-1}$.
This yields $K_{R,s}\leq\left(\begin{smallmatrix}n\\
s\end{smallmatrix}\right)\left(\begin{smallmatrix}R+s-1\\
s-1\end{smallmatrix}\right)\leq 2^n2^{R+s-1}$. 
From Lemma~\ref{gamma-easy}(a)(c), 
for each $x\in E_n\cap \Sigma$ satisfying
$V_n(x)(X\setminus\Gamma)=m/n$
there exist $R\in\{m,\ldots,2n\}$, $s\in\{1,\ldots,m\}$,
  $(n_1\cdots n_s,r_1\cdots r_s)\in K_{R,s}$ such that
$x\in \Delta_{n_1\cdots n_s}^{r_1\cdots r_s}$.
If $\delta\in (0,1/5]$ is sufficiently small, then
together with  Lemma~\ref{measure}
we obtain
\[\begin{split}\sum_{\stackrel{x\in E_n\cap \Sigma}{V_n(x)(X\setminus\Gamma)=m/n}   }\exp S_n\phi(x)&\leq
\sum_{s=1}^m\sum_{R=m }^{2n} \sum_{(n_1\cdots n_s,r_1\cdots r_s)\in K_{R,s}}\sum_{x\in \Delta_{n_1\cdots n_s}^{r_1\cdots r_s} }\exp S_n\phi(x)\\
&\leq 2^ne^{\gamma_0n}\sum_{s=1}^m
\sum_{R=m}^{2n} 2^{R+s-1}\delta^R\leq2^ne^{\gamma_0n}\sum_{R=m}^{2n} (4\delta)^R\\
&\leq\frac{2^ne^{\gamma_0n}(4\delta)^m}{1-4\delta}.\end{split}\]
From this and \eqref{step1}, the desired inequality
 \eqref{desired-ineq} follows and
the proof of Proposition~\ref{expo} is complete.
\end{proof}

\subsection{Verifying exponential tightness}\label{dedic}
We now use Proposition~\ref{expo} to show the exponential tightness.

\begin{prop}\label{e-tight}Let
$\phi\colon X\to\mathbb R$ be acceptable such that $P(\phi)<\infty$ and let $(\Sigma,\tau|_{\Sigma})$ be an induced system.
If there exists a local Gibbs state for the potential $\phi$ associated with $(\Sigma,\tau|_{\Sigma})$, 
then
$\{\tilde\mu_n\}_{n=1}^\infty$ is exponentially tight.
\end{prop}
\begin{proof}
The argument below is an
adaptation of the proof of Sanov's
theorem (see e.g., \cite{DemZei98}) to our setup.
For each integer $\ell\geq1$, we fix $\delta_\ell\in(0,1/5]$ 
such that 
\begin{equation}\label{theta}\frac{1}{1-4\delta_\ell}\sum_{m=1}^\infty e^{2\ell^2m}(4\delta_\ell)^m\leq1.\end{equation}
We 
 fix a non-decreasing integer sequence $\{N_i\}_{i=1}^\infty$ such that \eqref{N1} and \eqref{deltas} with $\delta=\delta_\ell$ hold.
We define a compact subset $\Gamma_\ell=\Gamma(\{N_i\}_{i=1}^\infty)$ of $X$ by \eqref{Gamma}, and set
\[\mathcal K^\ell=\left\{\nu\in\mathcal M(X)\colon \nu(\Gamma_\ell)\geq1-\frac{1}{\ell}\right\}.\]
Since $\mathcal M(X)$ is a Polish space and $\Gamma_\ell$ is a closed set, 
the weak* convergence
$\mu_k\to\mu$ for a sequence $\{\mu_k\}_{k=1}^\infty$ in $\mathcal K^\ell$
implies
$\limsup_{k\to\infty}\mu_k(\Gamma_\ell)\leq\mu(\Gamma_\ell)$.
Hence, $\mathcal K^\ell$ is a closed set.
For an integer $L\geq1$ we define
$$\mathcal K_L=\bigcap_{\ell=L}^\infty \mathcal K^\ell.$$
This set is tight, and by Prohorov's theorem any sequence in $\mathcal K_L$ has a limit point.
Hence $\mathcal K_L$ is sequentially compact.
Since the weak* topology is metrizable with the bounded Lipschitz metric \cite{DemZei98}, 
$\mathcal K_L$ is compact.
By Chebyshev's inequality, for $n\geq1$
we have
\[\begin{split}
\sum_{\stackrel{x\in E_n }{\exp\left(\ell^2n V_n(x)(X\setminus\Gamma_\ell)\right)\geq e^{\ell n}}}\exp S_n\phi(x )
\leq& e^{-2\ell n}\times\\
\sum_{\stackrel{x\in E_n}{V_n(x)(X\setminus\Gamma_\ell)\geq1/n} }&\exp\left(2\ell^2n V_n(x)(X\setminus\Gamma_\ell)\right)\exp S_n\phi(x)\\
=&e^{-2\ell n}\sum_{m=1}^{n}e^{2\ell^2m}
\sum_{\stackrel{x\in E_n }
{V_n(x)(X\setminus\Gamma) =m/n}    }\exp S_n\phi(x).
\end{split}\]
Combining this inequality with Proposition~\ref{expo} and \eqref{theta}, we have
\[\begin{split}
\sum_{\stackrel{x\in E_n }{V_n(x)\notin\mathcal K^\ell}}\exp S_n\phi(x)&=\sum_{\stackrel{x\in E_n }{V_n(x)(X\setminus\Gamma_\ell)\geq1/\ell}}\exp S_n\phi(x)
=\sum_{\stackrel{x\in E_n }{\exp\left(\ell^2n V_n(x)(X\setminus\Gamma_\ell)\right)\geq e^{\ell n}}}\exp S_n\phi(x )\\
&\leq \frac{2^nne^{\gamma_0n}}{1-4\delta_\ell}e^{-2\ell n}\sum_{m=1}^{n}e^{2\ell^2m}
(4\delta_\ell)^m\leq 2^nne^{\gamma_0n}e^{-2\ell n}.
\end{split}\]
If $L\geq1$ is large enough, then
\[\begin{split}\tilde\mu_n(\mathcal M(X)\setminus\mathcal K_L)&\leq\sum_{\ell=L}^\infty\tilde\mu_n(
\mathcal M(X)\setminus\mathcal K^\ell)=\sum_{\ell=L}^\infty\frac{1}{Z_n(\phi) }\sum_{\stackrel{x\in E_n }{V_n(x)\notin\mathcal K^\ell}}\exp S_n\phi(x)\\
&\leq\frac{2^nne^{\gamma_0n}}{Z_{n}(\phi)}\sum_{\ell=L}^{\infty} e^{-2\ell n}\leq \frac{2^ne^{\gamma_0n}}{Z_{n}(\phi)}e^{-Ln}.\end{split}\]
Combining this with the equality 
$\lim_{n\to\infty}(1/n)\log Z_n(\phi)=P(\phi)<\infty$
which follows from the acceptability of $\phi$, we obtain
\[\limsup_{n\to\infty}\frac{1}{n}\log \tilde\mu_n(\mathcal M(X )\setminus\mathcal K_L)\leq-L+\gamma_0+\log2.\]
Since $L\geq1$ can be made arbitrarily large, the exponential tightness of $\{\tilde\mu_n\}_{n=1}^\infty$ follows.
\end{proof}

\subsection{Existence of a local Gibbs state}\label{existence}
The next proposition ensures the existence of a local Gibbs state under the assumption of Theorem~A.
\begin{prop}\label{cylinder-m} 
Let
$\phi\colon X\to\mathbb R$, and let $(\Sigma,\tau|_{\Sigma})$ be an induced system for which
the induced potentials 
$\Phi_{\gamma}$, $\gamma\in\mathbb R$ associated with $\phi$ are locally H\"older continuous, 
and 
  there exists $\gamma_0\in\mathbb R$ such that $P(\Phi_{\gamma_0})=0$. 
There exists a $\theta$-invariant local Gibbs state for the potential $\phi$ associated with $(\Sigma,\tau|_{\Sigma})$.
\end{prop}
\begin{proof} Note that $P(\Phi_{\gamma_0}\circ\Pi)=P(\Phi_{\gamma_0})=0$. Since $\mathbb A^{\mathbb N}$ is the countable full shift, the finiteness of $P(\Phi_{\gamma_0}\circ\Pi)$ implies the summability of the  potential $\Phi_{\gamma_0}\circ\Pi$. By \cite[Corollary~2.7.5]{MauUrb03} together with 
the summability and the local H\"older continuity of $\Phi_{\gamma_0}\circ\Pi$, 
there exists a unique $\theta$-invariant Bowen's Gibbs state 
for the potential $\Phi_{\gamma_0}\circ \Pi$, which we denote by $\lambda_\phi$.
There exists $C\geq1$ such that
for every $m\geq1$, any
$\bold a\in{\mathbb A}^m$ and any
 $z\in[\![\bold a]\!]$ 
 we have
\[\begin{split}
C^{-1}\leq\frac{\lambda_\phi[\![\bold a]\!]}{\exp\left(-P(\Phi_{\gamma_0}\circ\Pi)m+\sum_{k=0}^{m-1}
\Phi_{\gamma_0  }\circ\Pi(\theta^k z)\right)}&\leq C.
\end{split}\]
For the series in the denominator of the fraction, for  $x\in\Pi[\![\bold a]\!]$ we have
\[\begin{split}\sum_{k=0}^{m-1}\Phi_{\gamma_0 }\circ\Pi(\theta^k\Pi^{-1}(x)) 
&=S_{\sum_{k=0}^{m-1}R(\tau^k x)}\phi( x)-\gamma_0\sum_{k=0}^{m-1}R(\tau^k x)\\
&=S_{|\bold a| }\phi( x)-
\gamma_0|\bold a|.\end{split}\]
Substituting this and $P(\Phi_{\gamma_0}\circ\Pi)=0$ into the denominator of the fraction implies that
 $\lambda_\phi$ is a local Gibbs state for the potential $\phi$.
\end{proof}
\begin{remark}\label{rem-equi}
Under the assumption and notation of Proposition~\ref{cylinder-m} and its proof, if $\gamma_0=P(\phi)$ 
and $\int R{\rm d}(\lambda_\phi\circ\Pi^{-1})<\infty$, then
the measure 
\[\frac{1}{\int R{\rm d}(\lambda_\phi\circ\Pi^{-1})}\sum_{n=0}^{\infty}
(\lambda_\phi\circ\Pi^{-1})|_{\{R >n\}}\circ\sigma^{-n}\]
is in $\mathcal M_\phi(X,\sigma)$, and it is an equilibrium state for the potential $\phi$. 
\end{remark}

\section{Proofs of the main results}
In this section we complete the proofs of all the theorems.
In Sections~\ref{low-sub} and \ref{key-upbd}, we prove lower and upper bounds
for certain fundamental open and closed subsets of $\mathcal M(X)$ respectively. In Section~\ref{pfthma}, we combine these bounds and the exponential tightness verified in Section~2 to complete the proof of Theorem~A.
In Section~\ref{pfthmc} we complete the proof of Theorem~B.
In view of applications, 
in Section~\ref{sufficience}
we give a sufficient condition for the vanishing of the pressure of the induced potential that is assumed in Theorem~A. 
Using this condition, we complete the proof of Theorem~C
 in Section~\ref{alt-proof}. In Section~\ref{sec-gen} we mention some generalizations.
\subsection{Lower bound for fundamental open sets}\label{low-sub}
We introduce notations in this and the next two subsections.
Let $C_u(X)$ denote the set of real-valued bounded uniformly continuous functions on $X$. For an integer $\ell\geq1$ we
define
\[
 C_u(X)^\ell=\{\vec\varphi=(\varphi_1,\ldots,\varphi_\ell)\colon\varphi_j\in C_u(X)\text{ for every }j\in\{1,\ldots,\ell\}\}.
\]
For $\vec{\varphi}=(\varphi_1,\ldots,\varphi_\ell)\in C_u(X)^{\ell}$, $\vec{\alpha}=(\alpha_1,\ldots,\alpha_\ell)\in\mathbb R^\ell$
and $\mu\in\mathcal M(X)$, 
 the expression 
$\int\vec{\varphi} {\rm d}\mu>\vec{\alpha}$  indicates that $\int\varphi_j {\rm d}\mu>\alpha_j$ holds for all $j\in\{1,\ldots,\ell\}$.
The meaning of $\int\vec{\varphi} {\rm d}\mu\geq\vec{\alpha}$ is analogous.
Put
$\|\vec \alpha\|=\max_{1\leq j\leq\ell}|\alpha_j|$.
For $\varepsilon\in\mathbb R$ we write 
$
\vec{\varepsilon}=(\varepsilon,\ldots,\varepsilon)\in\mathbb R^\ell.$
For $n\geq1$ and $p_1\cdots p_n\in\mathbb N^n$, let $\overline{p_1\cdots p_n}$ denote the element of $E_n$ that is contained in $[p_1\cdots p_n]\subset X$.

\begin{prop}\label{lowlem}
Let 
$\phi\colon X\to\mathbb R$ be acceptable and satisfy $P(\phi)<\infty$.  Let $\ell\geq1$, $\vec\varphi\in C_u(X)^\ell$ and $\vec\alpha\in\mathbb R^\ell$.
Let $\mathcal G\subset\mathcal M(X)$ be an open set of the form
\[\mathcal G=\left\{\mu\in\mathcal M(X)\colon \int\vec\varphi {\rm d}\mu>\vec\alpha\right\}.\]
For any measure $\mu\in\mathcal M_\phi(X,\sigma)\cap\mathcal G$, we have
\[\liminf_{n\to\infty}\frac{1}{n}\log
\tilde\mu_n(\mathcal G)\geq F_\phi(\mu).\]
\end{prop}
\begin{proof}
By virtue of the definition of the pressure $P(\phi)$, it suffices to show that
\begin{equation}\label{desired-low}\liminf_{n\to\infty}\frac{1}{n}\log
\sum_{\stackrel{x\in E_n}{V_n(x)\in\mathcal G} }\exp S_n\phi(x)\geq h(\mu)+\int\phi {\rm d}\mu.\end{equation}
The proof of \cite[Main~Theorem]{Tak} works verbatim to show the next lemma that
approximates non-ergodic measures with ergodic ones in a particular sense.

\begin{lemma}\label{nonergodic}
For any $\mu\in\mathcal M_\phi(X,\sigma)$ and any $\varepsilon>0$,
there exists an ergodic measure $\mu'\in\mathcal M_\phi(X,\sigma)$ which
is supported on a compact set and satisfies
\[|h(\mu)-h(\mu')|<\varepsilon,\
\left\|\int\vec\varphi {\rm d}\mu-\int\vec\varphi {\rm d}\mu'\right\|<\varepsilon
\ \text{ and }\ 
\left|\int\phi {\rm d}\mu-\int\phi {\rm d}\mu'\right|<\varepsilon.\]
\end{lemma}
\noindent 
By Lemma~\ref{nonergodic}, it suffices to show \eqref{desired-low}
for all $\mu\in\mathcal M_\phi(X,\sigma)$ which is ergodic.
Let $\varepsilon>0$ be such that \begin{equation}\label{qe0}\int\vec\varphi {\rm d}\mu>\vec\alpha+\vec\varepsilon.\end{equation}
Since each component of $\vec\varphi$ is bounded uniformly continuous and $\phi$ is acceptable,
one can use Birkhoff's ergodic theorem and Shannon-McMillan-Breiman's theorem to show that
for any sufficiently large $n\geq1$ there is
a finite subset $G_n$
of $\mathbb N^n$ such that
\begin{equation}\label{qe1}
 \left|\frac{1}{n}\log\#G_n- h(\mu)\right|<\frac{\varepsilon}{2},\end{equation}
 and for every $p_1\cdots p_n\in G_n$,
\begin{equation}\label{qe2-1}
\sup_{x\in[p_1\cdots p_n]}\left\|\int\vec\varphi {\rm d}V_n(x)- \int\vec\varphi {\rm d}\mu\right\|<\frac{\varepsilon}{2}\ \text{ and}\end{equation}
\begin{equation}\label{qe2-2} 
\sup_{x\in[p_1\cdots p_n]}\left|\frac{1}{n}S_n\phi(x)- \int\phi {\rm d}\mu\right|<\frac{\varepsilon}{2}.\end{equation}
Then \eqref{qe0} and \eqref{qe2-1} yield $\int\vec\varphi {\rm d}V_n(\overline{p_1\cdots p_n})>\vec\alpha$, and 
\eqref{qe2-2} yields
 $(1/n)S_n\phi(\overline{p_1\cdots p_n})>\int\phi {\rm d}\mu-\varepsilon/2.$
  Therefore
 \[\begin{split} \sum_{\stackrel{x\in E_n}{V_n(x)\in\mathcal G} }\exp S_n\phi(x)&\geq \sum_{p_1\cdots p_n\in G_n} \exp S_n\phi(\overline{p_1\cdots p_n})\geq \#G_n \exp\left(n\int\phi {\rm d}\mu-\frac{\varepsilon n}{2}\right).\end{split}\]
Taking logarithms, dividing by a sufficiently large $n$
and using \eqref{qe1} we have
\[\begin{split}\frac{1}{n}\log \sum_{\stackrel{ x\in E_n}{V_n(x)\in\mathcal G } }\exp S_n\phi( x)&\geq \frac{1}{n}\log\#G_n+ \int\phi {\rm d}\mu
-\frac{\varepsilon}{2}> h(\mu)+\int\phi {\rm d}\mu-\varepsilon.\end{split}\]
Letting $n\to\infty$ and then $\varepsilon\to0$ yields \eqref{desired-low}. The proof of Proposition~\ref{lowlem} is complete.
\end{proof}

\subsection{Upper bound for fundamental closed sets}\label{key-upbd}
We proceed to 
upper bounds on fundamental closed sets which are not necessarily compact. 
\begin{prop}\label{uplem}
Let 
$\phi\colon X\to\mathbb R$ be acceptable and satisfy $P(\phi)<\infty$. 
Let $\ell\geq1$, $\vec\varphi\in C_u(X)^\ell$, $\vec\alpha\in\mathbb R^\ell$
and let $\mathcal C\subset\mathcal M(X)$ be a non-empty closed set of the form
\[\mathcal C=\left\{\mu\in\mathcal M(X)\colon \int\vec\varphi {\rm d}\mu\geq\vec\alpha\right\}.\]
For any $\varepsilon>0$ there exists $\mu\in\mathcal M_\phi(X,\sigma)$ such that
$\int\vec\varphi {\rm d}\mu>\vec\alpha-\vec\varepsilon$ and
\[
\limsup_{n\to\infty}\frac{1}{n}\log\tilde\mu_n(\mathcal C)\leq F_\phi(\mu).\]
\end{prop}
\begin{proof}
A main ingredient is the next lemma, the proof of which is analogous to the standard proof of the variational principle \cite{Wal82}. 
For $n\geq1$ we put
\[D_n(\phi)=\sup_{p_1\cdots p_n\in \mathbb N^n}\sup_{x, y\in[p_1\cdots p_n]}S_n\phi(x)-S_n\phi(y).\]
 \begin{lemma}\label{horse}
For any $\varepsilon>0$
 there exists $n_0\geq1$ such that if $n\geq n_0$ then for any non-empty 
 finite subset $H_n$ of $\mathbb N^n$ satisfying
 $V_n(\overline{p_1\cdots p_n})\in\mathcal C$ for all $p_1\cdots p_n\in H_n$,
there exists a measure $\mu_0\in\mathcal M_\phi(X,\sigma)$ such that 
\[ \log\sum_{p_1\cdots p_n\in H_n}\sup_{[p_1\cdots p_n]}\exp S_n\phi\leq \left(h(\mu_0)+\int\phi {\rm d}\mu_0\right)n+D_n(\phi) \quad\text{and}\quad
    \int\vec\varphi {\rm d}\mu_0>\vec\alpha-\vec\varepsilon.\]
 \end{lemma}
 \begin{proof}
 Since all components of $\vec\varphi$ are bounded uniformly continuous functions on $X$, for any $\varepsilon>0$ there exists $n_0\geq1$ such that if $n\geq n_0$ then for any $p_1\cdots p_n\in\mathbb N^n$ satisfying $V_n(\overline{p_1\cdots p_n})\in\mathcal C$, 
 $\int\vec\varphi {\rm d}V_n(x)\geq\vec\alpha-(1/2)\vec\varepsilon$ holds for any
$x\in[p_1\cdots p_n]$.
In what follows we assume $n\geq n_0$.
 
   We set \[\Lambda=\bigcap_{k=0}^\infty  \sigma^{-nk}\left(\bigcup_{p_1\cdots p_n\in H_n}[p_1\cdots p_n]\right).\]
Then $\sigma^n|_\Lambda\colon\Lambda\to\Lambda$ is topologically 
 conjugate to the left shift acting on the finite full shift space \[H_n^{\mathbb N}=\{(q_m)_{m=1}^\infty\colon q_m\in H_n\text{ for every }m\geq1\}.\]
 Since the function $\hat\phi=S_n\phi$
 induces a continuous potential on $H_n^{\mathbb N}$, 
 the variational principle \cite{Wal82}
 yields
 \[
 \begin{split}\sup_{\hat\mu\in\mathcal M(\Lambda,\sigma^n|_\Lambda)}\left( h(\hat\mu)+\int \hat\phi {\rm d}\hat\mu\right)&=
 \lim_{m\to\infty}\frac{1}{m}\log\sum_{q_1\cdots q_m\in H_n^m}
\sup_{ [q_1\cdots q_m]} \left(\exp{\sum_{k=0}^{m-1}
 \hat\phi\circ\sigma^{nk} }\right),\end{split}\]
   where $\mathcal M(\Lambda,\sigma^n|_\Lambda)$ denotes 
 the space of $\sigma^n|_\Lambda$-invariant Borel probability measures
 endowed with the weak* topology,
 and $h(\hat\mu)$ denotes 
 the measure-theoretic entropy of $\hat\mu\in \mathcal M(\Lambda,\sigma^n|_\Lambda)$
 with respect to $\sigma^n|_\Lambda$.
For the right-hand side, we have
 \[
 \begin{split}
 \sum_{q_1\cdots q_m\in H_n^m}
\sup_{[q_1\cdots q_m]} \exp\left({\sum_{k=0}^{m-1}
 \hat\phi\circ\sigma^{nk}}\right)
&\geq\left(\sum_{p_1\cdots p_n\in H_n}\inf_{[p_1\cdots p_n]}\exp S_n\phi\right)^m\\
&\geq\left( \exp(-D_n(\phi))\sum_{p_1\cdots p_n\in H_n}
\sup_{[p_1\cdots p_n]}\exp S_n\phi\right)^m.
\end{split}\]
Taking logarithms of both sides, 
dividing by $m$ and then letting $m\to\infty$ gives
\[\lim_{m\to\infty}\frac{1}{m}\log
\sum_{q_1\cdots q_m\in H_n^m}
\sup_{[q_1\cdots q_m]} \exp\left(\sum_{k=0}^{m-1}
 \hat\phi\circ\sigma^{nk}\right)
\geq\log\sum_{p_1\cdots p_n\in H_n}
\sup_{[p_1\cdots p_n]}\exp S_n\phi-D_n(\phi).\]
Plugging this into the previous inequality yields
\[\sup_{\hat\mu\in\mathcal M(\Lambda,\sigma^n|_\Lambda)}\left( h(\hat\mu)+\int \hat\phi {\rm d}\hat\mu\right)\geq
\log\sum_{p_1\cdots p_n\in H_n}
\sup_{[p_1\cdots p_n]}\exp S_n\phi-D_n(\phi).\]
By the compactness of the space $\mathcal M(\Lambda,\sigma^n|_\Lambda)$ and the upper semicontinuity of the map $\hat\mu\mapsto  h(\hat\mu)+\int \hat\phi {\rm d}\hat\mu$ on this space, the supremum is attained, say by $\hat\mu_0$. 
The measure
$\mu_0 = (1/n)\sum_{j=0}^{n-1}\hat\mu_0\circ \sigma^{-j}$ is in $\mathcal M_\phi(X,\sigma)$
and satisfies 
\[\left(h(\mu_0)+\int\phi {\rm d}\mu_0\right)n=\sup_{\hat\mu\in\mathcal M(\Lambda,\sigma^n|_\Lambda)}\left( h(\hat\mu)+\int \hat\phi {\rm d}\hat\mu\right).\]
Since the support of $\mu_0$
is contained in set $\{x\in X\colon \int\vec\varphi {\rm d}V_n(x)>\vec\alpha-\vec\varepsilon/2\}$ by the choice of $n_0$ and the assumption $n\geq n_0$, we obtain $\int\vec\varphi {\rm d}\mu_0>\vec\alpha-\vec\varepsilon$ as required.
 \end{proof}
Continuing the proof of Proposition~\ref{uplem}, 
we note that the acceptability of $\phi$ and $P(\phi)<\infty$ implies $Z_n(\phi)<\infty$ for every $n\geq1$. Hence it is possible to choose
 a finite subset $H_n$ of the countable set
$\left\{p_1\cdots p_n\in\mathbb N^n\colon V_n(\overline{p_1\cdots p_n})\in \mathcal C\right\}$
such that
\[\sum_{\stackrel{p_1\cdots p_n\in\mathbb N^n}{V_n(\overline{p_1\cdots p_n})\in \mathcal C} }
\exp S_n\phi(\overline{p_1\cdots p_n})\leq 2\sum_{p_1\cdots p_n\in H_n}\exp S_n\phi(\overline{p_1\cdots p_n}).\]
By this inequality and Lemma~\ref{horse}, there exists $\mu_0\in\mathcal M_\phi(X,\sigma)$ such that
$\int\vec\varphi {\rm d}\mu_0>\vec\alpha-\vec\varepsilon$
and
\[\begin{split}
\log \sum_{\stackrel{x\in E_n }{V_n(x)\in\mathcal C} }
\exp S_n\phi(x)&=\log\sum_{\stackrel{p_1\cdots p_n\in\mathbb N^n}{V_n(\overline{p_1\cdots p_n})\in \mathcal C} }
\exp S_n\phi(\overline{p_1\cdots p_n})\\
&\leq\log\sum_{p_1\cdots p_n\in H_n }
\exp S_n\phi(\overline{p_1\cdots p_n})+\log2\\
&\leq  \log\sum_{p_1\cdots p_n\in H_n}\sup_{[p_1\cdots p_n]}\exp S_n\phi+\log2\\
&\leq \left(h(\mu_0)+\int\phi {\rm d}\mu_0\right)n+D_n(\phi)+\log2.
\end{split}\]
Since $\phi$ is acceptable, 
we have $D_n(\phi)=o(n)$ $(n\to\infty)$.
Dividing both sides of the last inequality by $n$, letting $n\to\infty$ and combining the result with
 $P(\phi)=\lim_{n\to\infty}(1/n)\log Z_n(\phi)$
 yields the desired inequality.
\end{proof}

\subsection{Proof of Theorem~A}\label{pfthma}
 Let
$\phi\colon X\to\mathbb R$ be acceptable and satisfy $P(\phi)<\infty$. Suppose there exists
an induced system for which
the induced potentials 
$\Phi_{\gamma}$, $\gamma\in\mathbb R$ associated with $\phi$ are locally H\"older continuous, 
and there exists $\gamma_0\in\mathbb R$ such that $P(\Phi_{\gamma_0})=0$.

Let $\mathcal G$ be a non-empty open subset of $\mathcal M(X)$.
Since subsets of $\mathcal M(X)$ of the form
$\left\{\mu\in\mathcal M(X)\colon\int\vec\varphi {\rm d}\mu>\vec\alpha\right\}$ 
with $\ell\geq1$, $\vec{\varphi}\in C_u(X)^{\ell}$, $\vec{\alpha}\in\mathbb R^\ell$
constitute a base of the weak* topology of $\mathcal M(X)$,
 $\mathcal G$ is written as the union 
$\mathcal G=\bigcup_{\lambda}\mathcal G_\lambda$ of sets of this form. 
For each $\mathcal G_\lambda$, Proposition~\ref{lowlem} gives
\[\liminf_{n\to\infty}\frac{1}{n}\log \tilde\mu_n(\mathcal G_\lambda)
\geq \sup_{\mathcal G_\lambda}F_\phi,\]
and
hence 
\[\liminf_{n\to\infty}\frac{1}{n}\log\tilde\mu_n(\mathcal G)
\geq\sup_{\lambda}\sup_{\mathcal G_\lambda} F_\phi= \sup_{\mathcal{G}} F_\phi= -\inf_{\mathcal G} I_\phi,
\]
as required in \eqref{LDPlow}.

Let $\mathcal C$ be a compact closed  
subset of $\mathcal M(X)$.
Let $\mathcal G$ be an arbitrary open set containing $\mathcal C$.
Since $\mathcal M(X)$ is metrizable by the bounded Lipschitz metric
and $\mathcal C$ is compact, we can choose $\varepsilon>0$ and finitely many closed sets $\mathcal{C}_1,\ldots,\mathcal C_s$ of the form 
$\mathcal C_k=\left\{\mu\in\mathcal M(X)\colon \int\vec\varphi_k {\rm d}\mu\geq\vec\alpha_k\right\}$
with $\ell_k\geq1$, $\vec\varphi_k\in C_u(X)^{\ell_k}$, $\vec\alpha_k\in\mathbb R^{\ell_k}$,
so that 
$\mathcal C \subset \bigcup_{k=1}^s \mathcal C_k \subset 
\bigcup_{k=1}^s \mathcal C_k(\varepsilon)\subset \mathcal G
$
where
$\mathcal C_k(\varepsilon)=\{\mu\in\mathcal M(X)\colon \int\vec\varphi_k {\rm d}\mu>\vec\alpha_k-\vec{\varepsilon}\}$.
By Lemma~\ref{horse} and 
$F_\phi\leq -I_\phi$, for $1\leq k\leq s$
 we have
\[\limsup_{n\to\infty}\frac{1}{n}\log\tilde\mu_n(\mathcal C_k)\leq -\inf_{\mathcal C_k(\varepsilon)}I_\phi
+\varepsilon.\]  
Then we have
\[
\limsup_{n\to\infty}\frac{1}{n}\log\tilde\mu_n(\mathcal C)
 \le \max_{1\le k\le s}  
\left(-\inf_{\mathcal C_k(\varepsilon)}I_\phi\right)+\varepsilon \le
-\inf_{\mathcal G}I_\phi+\varepsilon.
\]
Since $\varepsilon>0$ is arbitrary and
$\mathcal G$ is an arbitrary open set containing 
 $\mathcal C$, it follows that
\[
\limsup_{n\to\infty}\frac{1}{n}\log\tilde\mu_n(\mathcal C)\leq
\inf_{\mathcal G \supset \mathcal C} \left(-\inf_{\mathcal G} I_\phi\right)=
-\inf_{\mathcal C}I_\phi,\]
as required in \eqref{LDPup}. The last equality is due to the lower semicontinuity of $I_\phi$.

Since $\{\tilde\mu_n\}_{n=1}^\infty$ is exponentially tight by Proposition~\ref{e-tight}, by the standard arguments as in \cite{DemZei98}, 
the upper bound \eqref{LDPup} holds for any non-compact closed subset of $\mathcal M(X)$, and  $I_\phi$ is a good rate function.  This completes the proof of Theorem~A.\qed

\subsection{Proof of Theorem~B}\label{pfthmc}
 Let
$\phi\colon X\to\mathbb R$ be acceptable and satisfy $P(\phi)<\infty$. Suppose there exists
an induced system for which
the associated induced potentials 
$\Phi_{\gamma}$, $\gamma\in\mathbb R$ are locally H\"older continuous, 
and 
  there exists $\gamma_0\in\mathbb R$ such that $P(\Phi_{\gamma_0})=0$. 
  Assume that the minimizer of the rate function $I_\phi$ in \eqref{ratefcn} is unique, denoted by $\mu_{\rm min}$.
   Let $\{\tilde{\mu}_{n(j)}\}_{j=1}^\infty$
be an arbitrary convergent subsequence of $\{\tilde\mu_n\}_{n=1}^\infty$ with the limit measure $\tilde\mu$.
  It suffices to show that $\tilde\mu$ is the unit point mass at $\mu_{\min}$.
  
We fix a metric that generates the weak* topology on $\mathcal M(X)$. 
Since the rate function $I_\phi$ in \eqref{ratefcn} is the good rate function
by Theorem~A,
for any $\alpha>0$ the level set \[\mathcal L(\alpha)=\{\mu\in\mathcal M(X)\colon I_\phi(\mu)\leq\alpha\}\] is a compact set.
Let $\mu\in\mathcal M(X)\setminus\{\mu_{\rm min} \}$. By the lower semicontinuity
of the rate function and $I_\phi(\mu)>0$, there exists $r>0$ such that the closure of the 
 open ball $\mathcal B(r;\mu)$ of radius $r$ about $\mu$ in $\mathcal M(X)$
 does not intersect the level set $\mathcal L(I_\phi(\mu)/2)$.
The weak* convergence
$\tilde{\mu}_{n(j)}\to\tilde\mu$ gives
\[\tilde\mu(\mathcal B(r;\mu))\leq\liminf_{j\to\infty}\tilde\mu_{n(j)}(\mathcal B(r;\mu))\leq\limsup_{j\to\infty}\tilde\mu_{n(j)}(\overline{\mathcal B(r;\mu)}).\] By this and the large deviations upper bound for closed sets \eqref{LDPup},
we have
\[\tilde\mu(\mathcal B(r;\mu))\leq\limsup_{j\to\infty}\exp\left(-n(j)\inf_{\overline{\mathcal B(r;\mu)}}I_\phi \right)\leq\limsup_{j\to\infty}\exp\left(-\frac{n(j)}{2}I_\phi(\mu) \right)=0.\]
Hence, the support of $\tilde\mu$ does not contain $\mu$. Since
 $\mu$ is an arbitrary element of
 $\mathcal M(X)\setminus\{\mu_{\rm min} \}$, it follows that $\tilde\mu$ is the unit point mass at $\mu_{\min}$.
  This completes the proof of Theorem~B.
 \qed

\subsection{Sufficient condition for vanishing of pressure}\label{sufficience} 
A direct check of the condition $P(\Phi_{\gamma_0})=0$ in Theorem~A may be cumbersome, while checking the finiteness of induced pressures is considered to be easier.  
In view of applications, 
we give a sufficient condition for the second assumption in Theorem~A on the induced potential. 
\begin{lemma}\label{P-zero}Let
$\phi\colon X\to\mathbb R$, let
 $(\Sigma,\tau|_\Sigma)$ be an induced system and let
$\Phi_\gamma\colon \Sigma\to\mathbb R$ $(\gamma\in\mathbb R)$ be the associated family of induced potentials. If $\Phi_\gamma$, $\gamma\in\mathbb R$ are acceptable and
 there exists $\delta\in\mathbb R$ such that $0< P(\Phi_{\delta})<\infty$, then
  there exists $\gamma_0\in\mathbb R$ such that $P(\Phi_{\gamma_0})=0$.
 \end{lemma}
\begin{proof} 
 Let $n\geq1$.
 Since $|\bold a|\geq n$
 for any $\bold a\in\mathbb A^n$ 
 for any $\gamma>\delta$ we have
\[\sum_{\bold a\in\mathbb A^n}\sup_{[\bold a]}\exp(
 S_{| \bold a|}\phi-\gamma|\bold a|)\leq\exp(-(\gamma-\delta) n)\sum_{\bold a\in\mathbb A^n}\sup_{[\bold a]}\exp(
 S_{|\bold a|}\phi-\delta|\bold a|).\]
 Taking logarithms, dividing by $n$ and letting $n\to\infty$ yields
 $P(\Phi_{\gamma})\leq-\gamma+\delta+P(\Phi_{\delta})$. 
 Hence, the function
$\gamma\in[\delta,\infty)\mapsto P(\Phi_\gamma)$ is strictly decreasing.
 The variational principle \cite[Theorem~2.1.8]{MauUrb03} implies that this function is continuous.
Hence there exists $\gamma_0>\delta$ such that $P(\Phi_{\gamma_0})=0$.
 \end{proof}

\subsection{Proof of Theorem~C}\label{alt-proof}

Recall that the 
R\'enyi map $T$ in \eqref{renyi} is modeled on the countable full shift $X$ via the coding map $\pi$ in \eqref{pi}.
As in Remark~\ref{rem-renyi}, we set $p^*=2$ and
define an inducing scheme $(X^*,R)$ by
\eqref{xstar}, \eqref{def-R}, and
the induced system $(\Sigma,\tau|_{\Sigma})$ by  \eqref{sigmahat}, \eqref{xhat}. We also define the infinite alphabet
$\mathbb A$ and the coding map $\Pi$  by \eqref{ahat}, \eqref{codingmap} respectively, keeping the notation in Section~\ref{symbolic}. We have
 \[\Sigma=\left[\frac{1}{2},1\right)\cap\mathbb I\ \text{ and }\ \mathbb A=\left\{\bigcup_{q=2 }^\infty[p 1^{n-1} q]\colon p\geq2\text{ and } n\geq1\right\}.\]
 Recall that $\mathbb I$ denotes the set of irrational numbers in $(0,1)$.
 For a bounded interval $J\subset\mathbb R$, let $|J|$ denote its Euclidean length. 
Various positive constants which depend only on $T$ will be simply denoted by $C$.
 
For each
  $\bold a\in\mathbb A$ with $\bold a=\bigcup_{q=2 }^\infty[p 1^{n-1} q]$,
 we put 
\[J(\bold a)=T^{-1}\left(\left[\frac{1}{|\bold a|+1},\frac{1}{|\bold a|}\right)\right)\cap J_p.\] 
Then $R$ equals $|\bold a|$ on 
 the set $\Pi\bra\bold a\ket=\pi^{-1}(J(\bold a))$, where $\pi$ denotes the coding map in \eqref{pi}.
There is $C\geq1$ such that for any $\bold a\in W(\mathbb A)$ we have 
\begin{equation}\label{interval}C^{-1}\leq |J(\bold a)|\cdot|\bold a|^2\leq C.
\end{equation}
Define an induced map $U\colon \bigcup_{\bold a\in\mathbb A}
J(\bold a)\to [0,1)$ by $U|_{J(\bold a)}=T^{|\bold a|}|_{J(\bold a)}$.
 Recall that $\phi=-\log|T'\circ\pi|$.
For $\beta,\gamma\in\mathbb R$ define 
$\Phi_{\beta,\gamma}\colon \Sigma\to\mathbb R$ by
\[\Phi_{\beta,\gamma}( x)=\beta S_{R(x)}\phi(x)-\gamma R(x),\]
which is the induced potential associated with $\beta\phi$.



 \begin{lemma}\label{holder}
For all $\beta,\gamma\in\mathbb R$, $\Phi_{\beta,\gamma}$ is locally H\"older continuous with respect to the metric $d_\Sigma$.
\end{lemma}
\begin{proof}
From 
the bounded distortion near the neutral fixed point
\cite[Lemma~2.2]{Nak00}, 
there is $C>0$ such that
 for any $\bold a\in\mathbb A$ and all $x$, $y\in[\bold a]$ we have
\begin{equation}\label{holder-eq1}\Phi_{\beta,\gamma}(x)-\Phi_{\beta,\gamma}(y)\leq C\beta
|U(\xi)-U(\eta)|\leq C\beta,\end{equation}
where $\xi=\pi(x)$ and $\eta=\pi(y)$.
If $x\neq y$
then $d_\Sigma(x,y)=e^{-n}$ holds for some $n\geq2$,
and there exists
 $a_1\cdots a_n\in\mathbb A^n$ such that
$x,y\in\Pi\bra a_1\cdots a_n\ket$.
Since $\inf_{[0,1)\setminus J_1}|T'|\geq\rho$ holds for some $\rho>1$, if $n\geq3$ then we have
\begin{equation}\label{holder-eq2}|U(\xi)-U(\eta)|\leq \frac{|U^{n-1}(\xi )-U^{n-1}(\eta )|}{ \inf_{ \bigcap_{k=2}^{n-1}U^{-k}(J( a_{k}))}|(U^{n-2})'|}\leq \rho^{2-n}.\end{equation}
The local H\"older continuity of 
$\Phi_{\beta,\gamma}$ follows from \eqref{holder-eq1} and \eqref{holder-eq2}.
\end{proof}

\begin{lemma}\label{pr}
For any $\beta\in(1/2,1]$ there exists $\gamma\in\mathbb R$ such that  $0\leq P(\Phi_{\beta,\gamma})<\infty$.
\end{lemma}

\begin{proof}
From Lemma~\ref{holder}, $|T(J_p)|=1$ for $p\geq2$ and \eqref{interval}, there is $C>0$ such that for 
$\bold a=\bigcup_{q=2}^\infty[p1^{n-1}q]\in\mathbb A$
and all $\beta,\gamma\in\mathbb R$ we have
\[\frac{1}{|J_p|^\beta}\sup_{[\bold a]}\exp\Phi_{\beta,\gamma}=e^{-\gamma n }\sup_{[ \bold a]}\frac{\exp(\beta S_{n }\phi)}{|J_p|^\beta}\leq Ce^{-\gamma n }n^{-2\beta}.\]
 Summing the result over all $\bold a\in\mathbb A$, we have 
 \[
   \sum_{n=1}^\infty\sum_{\stackrel{\bold a\in\mathbb A}{|\bold a|=n}}\sup_{[\bold a]}\exp\Phi_{\beta,\gamma}\leq C \sum_{n=1}^\infty e^{-\gamma n}n^{-2\beta}\sum_{p=2}^\infty |J_p|^\beta\leq C \sum_{n=1}^\infty e^{-\gamma n}n^{-2\beta}\sum_{p=2}^\infty p^{-2\beta}.\]
 
Let
$\beta\in(1/2,1)$.
Then we have
 $P(\beta\phi)>0$ \cite{Iom10}, and 
  the above series
is finite for all $\gamma\in (0,P(\beta\phi)]$. In particular,
$P(\Phi_{\beta,P(\beta\phi)})$ is finite.
Since
any measure in $\mathcal M(X,\sigma)$ other than the unit point mass at $1^\infty=111\cdots$ charges $\Sigma$,
the equilibrium state $\mu_{\beta\phi}$ for the potential $\beta\phi$ satisfies $\mu_{\beta\phi}(\Sigma)>0$. Let $\hat\mu_{\beta\phi}$ denote the normalized restriction of $\mu_{\beta\phi}$ to $\Sigma$. Since $\tau$ is the first return map to $\Sigma$,
$\hat\mu_{\beta\phi}$ is $\tau|_{\Sigma}$-invariant and satisfies $\int R{\rm d}\hat\mu_{\beta\phi}<\infty$.
By the variational principle for $\Phi_{\beta,P(\beta\phi)}$ and
Abramov-Kac's formula 
\cite[Theorem~5.1]{Zwe05},
we obtain 
\[\begin{split}\infty>P(\Phi_{\beta,P(\beta\phi) })&\geq h(\hat\mu_{\beta\phi})+\int(\Phi_\beta-P(\beta\phi) R) {\rm d}\hat\mu_{\beta\phi}\\
&=
\left(F_{\beta\phi}(\mu_{\beta\phi})+P(\beta\phi)-P(\beta\phi) \right) \int  R {\rm d}\hat\mu_{\beta\phi}=0.\end{split}\]
We have verified\footnote{In fact, one can show $P(\Phi_{\beta,P(\beta\phi)})=0$. See \cite{PesSen08} for example.} that
$0\leq P(\Phi_{\beta,P(\beta\phi) })<\infty$ as required in Lemma~\ref{pr}.

For the remaining case $\beta=1$, we have
 $P(\phi)=0$ \cite{Iom10}.
From Lemma~\ref{holder} there is $C\geq1$ such that for  $n\geq1$ and $a_1\cdots a_n\in {\mathbb A}^n$,
\[C^{-1}\leq\frac{\left|\bigcap_{k=1}^{n}U^{-k}(J(a_{k}))\right|} 
{\sup_{[a_1\cdots a_n ]}\exp\left( \sum_{k=0}^{n-1}\Phi_{1,0 }\circ\tau^k\right)}\leq C.\]
Therefore, the sum of the denominator over all elements of
 $\mathbb A^n$ is uniformly bounded from both sides. 
This implies $P(\Phi_{1,0})=0$ as required in Lemma~\ref{pr}.
\end{proof}

Lemma~\ref{holder} and Lemmas~\ref{P-zero},
\ref{pr} altogether verify the assumption in Theorem~A
for the potential $\beta\phi$, $\beta\in(1/2,1]$.
It follows from \cite{Tak20} that for any $\beta\in(1/2,1]$,
any minimizer of the rate function $I_{\beta\phi}$ is an equilibrium state
for $\beta\phi$. Since the equilibrium state for $\beta\phi$ is unique \cite{Iom10},
the minimizer of the rate function $I_{\beta\phi}$ is unique. 
   Since the map $\pi$ in \eqref{pi} is continuous, 
   the assertions in Theorem~C follow from Theorems~A and B.
   \qed

     \subsection{Some generalizations}\label{sec-gen}
We have mainly worked on the two full shift spaces $X$ and $\Sigma$ (or $\mathbb A^\mathbb N$), the latter obtained from the former via inducing procedure as detailed in Section~\ref{symbolic}.
     The assumption that $X$ is the full shift 
has been used
 to construct sets of periodic points of the same period, in the proofs of exponential tightness (Lemma~\ref{measure}) and 
 the lower large deviation bound (Proposition~\ref{lowlem}).
 For the induced system  $(\Sigma,\tau|_\Sigma)$, we have effected 
 the thermodynamic formalism for countable Markov shifts \cite{MauUrb03}.

The setup in this paper can be slightly generalized.
The above-mentioned constructions of sets of periodic points can be done even if $X$ is replaced by a finitely primitive shift (see \cite{MauUrb03} for the definition). Then the induced shift space  becomes finitely irreducible, for which the thermodynamic formalism works too \cite{PU21}.

\subsection*{Acknowledgments}
I thank the referees for their careful readings of the manuscript and giving useful suggestions for improvements.
This research was partially supported by the JSPS KAKENHI 
19K21835 and 20H01811. 
\medskip

\noindent{\bf The conflict of interest statement:} We have no conflict of interest to disclose. \\

\noindent{\bf The date availability statement:} This article has no associated data.

\end{document}